\documentclass[12pt]{article}
\usepackage{amssymb,amsthm,amsmath,pb-diagram}
\usepackage{fullpage}
\newtheorem{thm}[equation]{Theorem}
\newtheorem{lemma}[equation]{Lemma}
\newtheorem{prop}[equation]{Proposition}
\newtheorem{cor}[equation]{Corollary}
\theoremstyle{definition}
\newtheorem{remark}[equation]{Remark}
\newtheorem{notation}[equation]{Notation}
\newtheorem{hypothesis}[equation]{Hypothesis}
\newcommand{\Gl}{\operatorname{GL}}
\newcommand{\GL}{\operatorname{GL}}
\newcommand{\Mat}{\operatorname{Mat}}
\newcommand{\Spec}{\operatorname{Spec}}
\newcommand{\N}{{\mathbb N}}
\newcommand{\m}{{\mathfrak m}}

\newcommand{\A}{{\mathbb A}}
\newcommand{\C}{{\mathcal C}}
\renewcommand{\O}{{\mathcal O}}
\newcommand{\F}{\mathcal{F}}
\renewcommand{\cong}{\equiv}
\renewcommand{\mod}{\operatorname{mod}}
\renewcommand{\P}{{\mathbb P}}
\newcommand{\Vect}{\operatorname{Vect}}
\newcommand{\iso}{\to^{\!\!\!\!\!\!\!\sim\,}}
\newcommand{\End}{\operatorname{End}}
\newcommand{\Hom}{\operatorname{Hom}}
\newcommand{\Gal}{\operatorname{Gal}}
\newcommand{\Ind}{\operatorname{Ind}}
\newcommand{\Br}{\operatorname{Br}}
\newcommand{\ch}{\operatorname{char}}
\newcommand{\V}{\operatorname{V}}
\newcommand{\DMod}{\operatorname{\partial-Mod}}
\newcommand{\PP}{\operatorname{PP}}

\numberwithin{equation}{section}
\begin{document}
\title{Patching over Fields\\
}
\author{David Harbater\footnote{Supported in part by NSF Grant DMS-0500118.} \and Julia Hartmann\footnote{Supported by the German National Science Foundation (DFG).}}
\date{}
\maketitle
\begin{abstract}
\noindent We develop a new form of patching that is both far-reaching and more elementary than the previous versions that have been used in inverse Galois theory for function fields of curves.  A key point of our approach is to work with fields and vector spaces, rather than rings and modules.
After presenting a self-contained development of this form of patching, we obtain applications to other structures such as Brauer groups and differential modules.  
\end{abstract}

%%%%%%%%%%%%%%%%%%%%%%%%%%%%%%%%%%%%%%%%%%%%%%%%%%%%%%%%%%%%%%%%%%%%%%%%%%%%%%%%
\section{Introduction}
%%Section 1
%%%%%%%%%%%%%%%%%%%%%%%%%%%%%%%%%%%%%%%%%%%%%%%%%%%%%%%%%%%%%%%%%%%%%%%%%%%%%%%%

This manuscript introduces a new form of {\it patching}, a method that has been used to prove results in Galois theory over function fields of curves (e.g.\ see the survey in \cite{MSRI}).  Our approach here, which involves patching vector spaces given over a collection of fields, is more elementary than previous approaches, while facilitating various applications.

There are several forms of patching in the Galois theory literature, all drawing inspiration from ``cut-and-paste'' methods in topology and analysis, in which spaces are constructed on metric open sets and glued on overlaps.  Underlying this classical approach are Riemann's Existence Theorem (e.g.\ see \cite{MSRI}, Theorem~2.1.1), Serre's GAGA \cite{gaga}, and Cartan's Lemma on factoring matrices \cite{cartan}.  In the case of formal patching (e.g.\ in \cite{GCAL}, \cite{harbaterstevenson}, \cite{pries}), one considers rings of formal power series, and ``patches'' them together using Grothendieck's Existence Theorem on sheaves over formal schemes (\cite{EGA}, Corollary 5.1.6).  In the context of rigid patching (e.g.\ in \cite{liu}, \cite{raynaud}, \cite{popep}), one relies on Tate's rigid analytic spaces, where there is a form of ``rigid GAGA'' that takes the place of Grothendieck's theorem.  The variant known as algebraic patching (e.g.\ \cite{haranvoelklein}, \cite{voelklein}, \cite{haranjarden2}) restricts attention to the line, and draws on ideas from the rigid approach (most notably, convergent power series rings).  But that strategy avoids relying on more substantial geometric results, and instead works with normed rings and versions of Cartan's Lemma.

The current approach differs from formal and rigid patching by focusing on vector spaces rather than modules; i.e.\ by working over (fraction) fields rather than rings.  Doing so makes it possible for us to prove our patching results more directly, without the more substantial foundations needed in the other approaches.  As a result, our approach is conceptually simpler and should be more accessible to those not already familiar with patching methods.
Moreover our method is not restricted to the case of the line.  In addition to providing a framework in which one can prove the sort of results on inverse Galois theory that have been shown using previous methods (see Section~\ref{finitegroups} below), our approach
also permits easy application of patching to other situations in which one works just with fields and not with rings.  See Section~\ref{algebras} for an application to Brauer groups of fields, and Section~\ref{differential} on an application to differential modules (which are in fact vector spaces).  Further applications in these directions appear in \cite{hhk} and \cite{harhar2}.

A framework for stating patching results can be found in Section~\ref{setup}, followed by some preliminary results in Section~\ref{preliminarysection}.  Our main patching result (Theorem~\ref{globalpatching}) and a variant (Theorem~\ref{globalpatchingseveral}) are shown in Section~\ref{globalsection}.  Section~\ref{localsection} takes up related forms of patching, in which ``more local'' patches are used; the main result there is Theorem~\ref{clpatching}, along with a variant, Theorem~\ref{localpatchingseveral}.  A further generalization to singular curves appears in Section~\ref{singular}.  The versions in Section~\ref{localsection} and~\ref{singular} are designed to allow the method of patching over fields to be used in a variety of future applications.  
(Those interested just in our main form of patching, Theorem~\ref{globalpatching}, can skip Sections~\ref{localsection} and   \ref{singular}, as well as the second half of Section~\ref{setup}.)
Finally, in Section~\ref{applications}, we show how our new version of patching can be used to prove both old and new results.

We thank Daniel Krashen for helpful discussions concerning the
application to Brauer groups in Section~\ref{algebras}, 
and both him and Moshe Jarden for further comments and suggestions.  We also thank the Mathematical Sciences Research Institute for their hospitality during the writing of this paper. 

%%%%%%%%%%%%%%%%%%%%%%%%%%%%%%%%%%%%%%%%%%%%%%%%%%%%%%%%%%%%%%%%%%%%%%%%%%%%%%%%
\section{The setup for patching over fields} \label{setup}
%%Section 2.
%%%%%%%%%%%%%%%%%%%%%%%%%%%%%%%%%%%%%%%%%%%%%%%%%%%%%%%%%%%%%%%%%%%%%%%%%%%%%%%%

The general framework for patching can be expressed in a categorical language that permits its use in various contexts.  Here we provide such a framework for patching vector spaces over fields; later in Section~\ref{applications}, we show how our results can be extended and applied to patching other objects over fields.  We begin with some notation.

If $\alpha_i:\C_i \to \C_0$ are functors ($i=1,2$), then we may form the {\bf
2-fibre product}  category $\C_1 \times_{\C_0} \C_2$ (with respect to $\alpha_1, \alpha_2$), defined as follows:  An
object in the category consists of a pair $(V_1,V_2)$ together with an isomorphism
$\phi:\alpha_1(V_1) \iso
\alpha_2(V_2)$ in $\C_0$, where $V_i$ is an object
in $\C_i$  ($i=1,2$).  A morphism from $(V_1,V_2;\phi)$ to
$(V_1',V_2';\phi')$ consists of  morphisms $f_i:V_i \to V_i'$ in $\C_i$ (for
$i=1,2$)  such that $\phi' \circ \alpha_1(f_1) = \alpha_2(f_2) \circ \phi$.

For any field $F$, we write $\Vect(F)$ for the category of finite dimensional $F$-vector 
spaces. If $F_1, F_2$ are subfields of
a field $F_0$ and we let $\C_i = \Vect(F_i)$, then
there are base change functors $\alpha_i:\C_i \to \C_0$
given on objects by $\alpha_i(V_i) = V_i \otimes_{F_i} F_0$.  So we can form the
category $\C := \Vect(F_1)\times_{\Vect(F_0)} \Vect(F_2)$ with respect to these functors (and in the sequel, the functors $\alpha_i$ will be understood, though suppressed in the $2$-fibre product notation). 

Given an object $(V_1,V_2;\phi)$ in the above category $\C$, let $V_0 = \alpha_2(V_2) = V_2 \otimes_{F_2} F_0$.  Then $V_0, V_1, V_2$ are each vector spaces over $F := F_1 \cap F_2 \le F_0$.  
Let $i_2:V_2 \hookrightarrow V_0$ be the natural inclusion, and let $i_1:V_1
\hookrightarrow V_0$ be the composition of the natural inclusion $V_1
\hookrightarrow \alpha_1(V_1)  = V_1 \otimes_{F_1} F_0$ with $\phi$.  With respect to the inclusions $i_1,i_2$, we may form the vector space fibre product $V := V_1 \times_{V_0} V_2 = \{(v_1,v_2) \in V_1 \times V_2 \,|\,i_1(v_1)=i_2(v_2)\}$ over $F$; we call $V$ the {\bf fibre product} of the object $(V_1,V_2;\phi)$.  
Note that
if we identify  $V_1$ (resp.\ $V_2$) with its image under $i_1$ (resp.\ $i_2$),
then the fibre product $V$ is just the {\it intersection} of $V_1$ and $V_2$ inside
$V_0$. Of course this identification depends on $\phi$ (since $i_1$ depends on $\phi$).

The following is a special case of \cite{CAPS}, Proposition~2.1.

\begin{prop}\label{harb}
Let $F_1,F_2\leq F_0$ be fields, and let $F=F_1\cap F_2$. 
Let 
$$\beta:\Vect(F)\rightarrow \Vect(F_1)\times_{\Vect(F_0)}\Vect(F_2)$$ 
be the natural map given by base change.
Then the following two statements are equivalent:
\begin{enumerate}
\item $\beta$ is an equivalence of categories.\label{harb1}
\item For 
every positive integer $n$ and 
every matrix $A\in \Gl_n(F_0)$ there
  exist matrices $A_i\in \Gl_n(F_i)$ such that $A=A_1A_2$.\label{harb2}
\end{enumerate}
Moreover if these conditions hold, then the inverse of $\beta$ (up to isomorphism) is given on objects by taking the fibre product.
\end{prop}

Our main Theorems~\ref{globalpatching} and \ref{clpatching} assert that the base change functor $\beta:\Vect(F)\rightarrow \Vect(F_1)\times_{\Vect(F_0)}\Vect(F_2)$ is an equivalence of categories in situations where certain fields $F \le F_1, F_2 \le F_0$ arise geometrically.  The above proposition reduces the proofs there to showing the following two statements in those contexts:
\begin{itemize}
\item[]
\hspace{-.5cm}{\it Factorization}: For 
every 
$A\in \Gl_n(F_0)$ there exist 
$A_i\in \Gl_n(F_i)$ such that $A=A_1A_2$.\label{factcondition}
\item[]\hspace{-.55cm}
{\it Intersection}:  $F_1\cap F_2=F$.\label{intersectcondition}
\end{itemize}
In Sections~\ref{globalsection} and \ref{localsection}, we prove each of these two conditions in turn, and as a result obtain the main theorems.  Beforehand, in Section~\ref{preliminarysection}, we prove a factorization result that will be useful in proving both of the above two conditions.  In later results (Theorems~\ref{localpatchingseveral} and \ref{singularpatching}), we will consider a more general type of situation, and for this we introduce the following definitions (which 
give another perspective on the results of this paper, though they 
are not otherwise essential and may be skipped on a first reading).  

Let $\F := \{F_i\}_{i\in I}$ be a finite inverse system of fields (not necessarily filtered), whose inverse limit is a field $F$.  Let $\iota_{ij}:F_i \to F_j$ denote the inclusion map associated to $i,j \in I$ with $i \succ  j$ in the partial ordering on the index set $I$.  By a (vector space) {\bf patching problem} for the system $\F$ we will mean a system $\mathcal{V} := \{V_i\}_{i \in I}$ of finite dimensional $F_i$-vector spaces for $i \in I$, together with $F_i$-linear maps $\nu_{ij}:V_i \to V_j$ for all $i\succ j$ in $I$, such that for $i \succ j$ in $I$, the induced $F_j$-linear map $\nu_{ij} \otimes_{F_i} F_j: V_i \otimes_{F_i} F_j \to V_j$ is an isomorphism.  Note that for any patching problem $\mathcal{V}$, the dimension ${\rm dim}_{F_i} V_i$ is independent of $i \in I$; and we call this the {\bf dimension} of the patching problem, denoted ${\rm dim}\, \mathcal{V}$.

A {\bf morphism of patching problems} $\{V_i\}_{i \in I} \to \{V_i'\}_{i \in I}$ for $\F$ is a collection of $F_i$-linear maps $\phi_i:V_i \to V_i'$ (for $i \in I$) which are compatible with the maps $\nu_{ij}:V_i \to V_j$ and $\nu_{ij}':V_i' \to V_j'$.  The patching problems for $\F$ thus form a category $\PP(\F)$.  (One can also consider the analogous notion of {\it algebra patching problems}, in which each of the finite dimensional vector spaces is given the structure of an associative algebra over its base field.  Similarly one can consider patching problems for (finite dimensional) commutative algebras, central simple algebras, etc.  These also form categories.)

Every finite dimensional $F$-vector space $V$ induces a patching problem $\beta(V)$ for $\F$, by taking $V_i = V \otimes_F F_i$ and taking $\nu_{ij} = {\rm id}_V \otimes_F \iota_{ij}$.  Here $\beta$ defines a functor from the category $\Vect(F)$ of finite dimensional $F$-vector spaces to the category $\PP(\F)$.  If $\mathcal V$ is a patching problem for $\F$, and $\beta(V)$ is isomorphic to $\mathcal V$, we say that $V$ is {\bf solution} to the patching problem $\mathcal V$.  If $\beta: \Vect(F) \to \PP(\F)$ is an equivalence of categories, then it defines a bijection on isomorphism classes of objects; i.e., every patching problem for $\F$ has a unique solution up to isomorphism.

The situation described in Proposition~\ref{harb} above can then be rephrased in terms of patching problems for the inverse system $\F := \{F_0, F_1, F_2\}$ with $0 \prec 1,2$ in the partial ordering, and with corresponding inclusions $\iota_{i0}:F_i \to F_0$ for $i=1,2$.  Namely, the proposition says that in this situation, the above functor $\beta$ is an equivalence of categories if and only if the matrix factorization condition~(\ref{harb2}) of the proposition holds.  As noted above, every patching problem $\{V_0,V_1,V_2\}$ for $\F$ then has a unique solution $V$ up to isomorphism; and by the last assertion in the proposition, $V$ is given by the fibre product $V_1 \times_{V_0} V_2$, or equivalently by the inverse limit of the finite inverse system $\{V_i\}$.  As also noted above, with respect to the inclusions of $V_1, V_2$ into $V_0$, we may also regard this fibre product as the intersection $V_1 \cap V_2$ in $V_0$.
The above result then has the following corollary:

\begin{cor}\label{dimcriterion}
Let $F_1,F_2 \leq F_0$ be fields and write $F=F_1\cap F_2$. Let
  $\mathcal{V} = \{V_i\}$ be a patching problem for $\F := \{F_i\}$, and let
  $V=V_1 \cap V_2$ inside $V_0$.  
Then the patching problem $\mathcal{V}$ has a solution if and only if $\dim_F
  V=\dim\,\mathcal{V}$; and in this case, $V$ is a solution.
\end{cor}

\begin{proof}
If there is a solution $V'$ to the patching problem, then $V' = V_1 \cap V_2$ inside $V_0$ by Proposition~\ref{harb} as rephrased above in terms of patching problems, using the identification $V_1 \times_{V_0} V_2 = V_1 \cap V_2$.  Here ${\rm dim}_F V = {\rm dim}_{F_i} V_i = {\rm dim}\, \mathcal{V}$ since $V \otimes_F F_i$ is $F_i$-isomorphic to $V_i$.  

Conversely, if $\dim_F V = \dim\, \mathcal{V}$, then $\dim_{F_i} (V \otimes_F F_i) = \dim_F V = \dim_{F_i} V_i$; so the inclusion $V \otimes_F F_i
\hookrightarrow V_i$ induced by the natural map $V \hookrightarrow V_i$ is an
isomorphism of $V_i$-vector spaces for $i=0,1,2$.  Since $V$ is a fibre product, the 
three maps $V \hookrightarrow V_i$ are compatible; and so $V$ (together with these inclusions) is a
solution to the patching problem. 
\end{proof}

Concerning the last assertion of Proposition~\ref{harb}, we have the following more general result:

\begin{prop}\label{patchprob}
Let $\F = \{F_i\}_{i\in I}$ be a finite inverse system of fields whose inverse limit is a field $F$, and let $\mathcal{V} := \{V_i\}_{i \in I}$ be a patching problem for $\F$, with a solution $V$.  Then $V$ and the associated system of isomorphisms $V \otimes_F F_j \iso V_j$ (for $j \in I$) can be identified with the inverse limit $\displaystyle\lim_{\leftarrow}\, V_i$ (as $F$-vector spaces) along with the maps $\displaystyle\bigl(\lim_{\leftarrow}\, V_i \bigr) \otimes_F F_j\to V_j$.
\end{prop}

\begin{proof}
This is immediate from the $F$-vector space identity 
$\displaystyle V \otimes_F \bigl( \lim_\leftarrow F_i\bigr) = \lim_\leftarrow \bigl(  V \otimes_F F_i \bigr)$.
\end{proof}

A special case is that the index set $I$ of the inverse system is of the form $\{0,1,\dots,r\}$ with the partial ordering of $I$ given by $i \succ 0$ for $i=1,\dots,r$ and with no other order relations (as in Proposition~\ref{harb}, where $r=2$).  Then the above inverse limits can be interpreted as fibre products: 
$$
F = F_1 \times_{F_0} F_2 \times_{F_0} \cdots \times_{F_0} F_r,
\hskip .3in V = V_1 \times_{V_0} V_2 \times_{V_0} \cdots \times_{V_0} V_r.
$$
If we identify each field and each vector space with its image under the respective inclusion, we can also regard $F$ as the intersection of $F_1,\dots,F_r$ inside $F_0$ and similarly for $V$ and the $V_i$, generalizing the context of Proposition~\ref{harb} above. (This situation arises in Theorem~\ref{globalpatchingseveral} below.)  

Another special case of the above set-up arises in the context of Theorem~\ref{singularpatching} below.  Consider an 
index set $I$ of the form $I_0 \cup I_1 \cup I_2$, where the partial ordering has the property that for every $i_0 \in I_0$ there are unique elements $i_1 \in I_1$ and $i_2 \in I_2$ such that $i_1 \succ i_0$ and $i_2 \succ i_0$; and where there are no other relations in I.  For an inverse system of fields $\F = \{F_i\}$ indexed by $I$, giving a patching problem (up to isomorphism) for $\F$ is equivalent to giving a finite dimensional  $F_i$-vector space $V_i$ for each $i \in I_1 \cup I_2$, and giving an $F_{i_0}$-vector space isomorphism 
$\mu_{i_1,i_2,i_0}:V_{i_1} \otimes_{F_{i_1}} F_{i_0} \iso V_{i_2} \otimes_{F_{i_2}} F_{i_0}$ 
for every choice of $i_1,i_2,i_0$ such that each $i_j \in I_j$ and $i_1,i_2 \succ i_0$.  Namely, given a patching problem for $\F$, the collection of maps $\nu_{i_1,i_0}$ and $\nu_{i_2,i_0}$ (for $i_1,i_2 \succ i_0$) determines a collection of maps $\mu_{i_1,i_2,i_0}$ as above, given by $(\nu_{i_2,i_0} \otimes_{F_{i_2}} F_{i_0})^{-1}\circ (\nu_{i_1,i_0} \otimes_{F_{i_1}} F_{i_0})$.  Conversely, given $F_i$-vector spaces $V_i$ for $i \in I_1 \cup I_2$ and a collection of maps $\mu_{i_1,i_2,i_0}$ as above, for $i_0 \in I_0$ we may define $V_{i_0} := V_{i_2} \otimes_{F_{i_2}} F_{i_0}$ where $i_2$ is the unique index in $I_2$ with $i_2 \succ i_0$.  We can then let 
$\nu_{i_2,i_0}$ be the natural inclusion $V_{i_2} \hookrightarrow V_{i_2} \otimes_{F_{i_2}} F_{i_0} = V_{i_0}$; and let 
$\nu_{i_1,i_0}$ be the composition of the natural inclusion
$V_{i_1} \hookrightarrow V_{i_1} \otimes_{F_{i_1}} F_{i_0}$ with 
$\mu_{i_1,i_2,i_0}$.  In this way we obtain inverse transformations between families $\{\mu_{i_1,i_2,i_0}\}$ and families $\{\nu_{i_1,i_0}\}$, thereby establishing the asserted equivalence.

%%%%%%%%%%%%%%%%%%%%%%%%%%%%%%%%%%%%%%%%%%%%%%%%%%%%%%%%%%%%%%%%%%%%%%%%%%%%%%%%
\section{Preliminary results} \label{preliminarysection}
%%Section 3
%%%%%%%%%%%%%%%%%%%%%%%%%%%%%%%%%%%%%%%%%%%%%%%%%%%%%%%%%%%%%%%%%%%%%%%%%%%%%%%%

%-------------------------------------------------------------------------------
\subsection{Matrix factorization}

Below we show two matrix factorization results that will be used in proving
our main results, Theorems~\ref{globalpatching} and \ref{clpatching}.  We begin with a lemma that reduces the problem to factoring matrices that are close to the identity.  This reduction parallels the strategy employed in \cite{voelklein}, Section~11.3, and \cite{haranjarden}, Section~4.

\begin{lemma}\label{conditionalfactorization} 
Let $\hat{R}_0$ be a complete discrete valuation ring 
with uniformizer $t$,
and let $\hat{R}_1,\hat{R}_2\leq \hat{R}_0$ be $t$-adically complete subrings that contain $t$.
Write $F_0, F_1$ for the
fraction fields of $\hat R_0, \hat R_1$, and assume that $\hat R_1/t\hat R_1$ is a domain whose fraction field 
equals $\hat R_0/t\hat R_0$.  
\renewcommand{\theenumi}{\alph{enumi}}
\begin{enumerate}
\item 
Then $R_0 := \hat R_0 \cap F_1 \subset F_0$ is $t$-adically dense in $\hat R_0$. \label{density}
\item
Suppose that for each $A \in \Gl_n(\hat R_0)$
satisfying $A \equiv I \; (\mod\, t\Mat_n(\hat R_0))$, there  exist $A_1 \in \Gl_n(F_1)$,
$A_2 \in \Gl_n(\hat R_2)$ such that $A=A_1A_2$.
Then the same conclusion holds for all matrices $A \in \Mat_n(\hat R_0)$ with non-zero determinant. \label{factoring}
\end{enumerate}
\end{lemma}

\begin{proof}
(\ref{density})
To prove this, we will show by induction that for every $f \in \hat R_0$ and every $m \ge 0$, there is an element $f_m \in R_0$ such that $f-f_m \in t^m\hat R_0$.  This is trivial for $m=0$, taking $f_m=0$.  Suppose the assertion holds for $m-1$, and write $f-f_{m-1}=t^{m-1}e$, with $e \in \hat R_0$.  The reduction $\bar e \in \hat R_0/t\hat R_0$ modulo $t\hat R_0$ lies in the fraction field of $\hat R_1/t\hat R_1$, and so may be written as $\bar g/\bar h$, with $\bar g, \bar h \in \hat R_1/t\hat R_1$ and $\bar h \ne 0$.  Pick $g,h \in \hat R_1$ that reduce to $\bar g, \bar h$ modulo $t\hat R_1$.  Since $\hat R_0/t\hat R_0$ is a field, $\bar h$ is a unit there, and so $h$ is a unit in the $t$-adically complete ring $\hat R_0$ (which is the valuation ring of $F_0$).  Thus $g/h \in \hat R_0$, and $e-g/h \in t\hat R_0$.  Taking $f_m = f_{m-1} + t^{m-1}g/h \in \hat R_0 \cap F_1 = R_0$, we have $f-f_m \in t^m\hat R_0$, proving part~(\ref{density}).

(\ref{factoring})  
Let $A \in \Mat_n(\hat R_0)$ be a matrix with non-zero determinant.  So $A^{-1} \in t^{-r}\Mat_n(\hat R_0) \subset \Mat_n(F_0)$ for some $r \ge 0$.  Since $R_0$ is $t$-adically dense in $\hat R_0$ by part (\ref{density}),
there is a $C_0 \in \Mat_n(R_0)$ that is congruent to $t^rA^{-1}\in \Mat_n(\hat R_0)$ modulo $t^{r+1}\Mat_n(\hat R_0)$.  Let $C = t^{-r}C_0 \in t^{-r}\Mat_n(R_0)\subset \Mat_n(F_1)$.  Then $C - A^{-1} \in t\Mat_n(\hat R_0)$ and so $CA - I \in t\Mat_n(\hat R_0)$.  Hence $CA \in \GL_n(\hat R_0)$, and in particular, $C$ has non-zero determinant; i.e.\ $C \in \GL_n(F_1)$.  By hypothesis, there exist $A_1' \in \GL_n(F_1)$, $A_2 \in \GL_n(\hat R_2)$ such that $CA=A_1'A_2$ in $\GL_n(F_0)$.  Let $A_1 = C^{-1}A_1' \in \GL_n(F_1)$.  Then $A = A_1A_2$, as asserted. 
\end{proof}

Lemma~\ref{conditionalfactorization} will be used in conjunction with the following proposition, which provides a condition under which the factorization hypothesis of the above lemma is satisfied.

\begin{prop}\label{factorizationcloseto1}
Let $T$ be a complete discrete valuation ring 
with uniformizer $t$, let $\hat{R}_0$ be a $t$-adically complete $T$-algebra which is a domain,
and let $\hat{R}_1,\hat{R}_2\leq \hat{R}_0$ be $t$-adically complete subrings containing $T$, with fraction fields $F_i$ ($i=0,1,2$).
Assume that $M_1 \subset F_0$ is a $t$-adically complete (e.g.\ finitely generated)
$\hat{R}_1$-submodule of $\hat R_0 \cap F_1$ such that for every $a\in \hat{R}_0$, there exist $a_1\in M_1$ and $a_2\in \hat{R}_2$ 
for which $a\cong a_1+a_2 \;(\mod t\hat R_0)$.
Then every 
$A\in \Gl_n(\hat{R}_0)$ with $A\cong I \;(\mod\, t\Mat_n(\hat R_0))$
can be written as $A=A_1A_2$ with $A_1\in
\Mat_n(M_1)$  and $A_2\in \Gl_n(\hat{R}_2)$.
Necessarily, $A_1 \in \Gl_n(F_1)$.
\end{prop}

\begin{proof}
The proof proceeds by constructing $A_1$ and $A_2$, respectively, as the limits of a sequence of matrices $B_i$ with coefficients in 
$M_1$, and a sequence of matrices $C_i$ with coefficients in $\hat{R}_2$, such that   
\begin{alignat*}{2}
B_0&=C_0= &&I,\\
A&\cong B_iC_i \; &&(\mod t^{i+1}\Mat_n(\hat R_0)),\\
B_i&\cong B_{i-1} \; &&(\mod t^{i}\Mat_n(M_1)),\\
C_i&\cong C_{i-1} \; &&(\mod t^{i}\Mat_n(\hat R_2)).
\end{alignat*}
By $t$-adic completeness, these limits exist and $A_2 \in \Gl_n(\hat R_2)$ since $A_2 \equiv I \; (\mod t\hat R_2)$.  Also, 
$A_1 \in \GL_n(F_1)$ because $M_1 \subset F_1$ and since $A,A_2$ have non-zero determinant.

We now construct this sequence inductively.  
So suppose for some $n \ge 1$ and for all $i \leq n-1$
that $B_i, C_i$ have already been constructed, satisfying the above conditions; and we wish to construct $B_n, C_n$.

By the inductive hypothesis, 
$$A-B_{n-1}C_{n-1}=t^n\tilde{A}_n$$
for some $\tilde{A}_n$ with coefficients in $\hat{R}_0$.
By the hypothesis of the proposition (applied to the entries of $\tilde A_n$), there exist matrices $B_n'\in \Mat_n(M_1)$ and  $C_n' \in \Mat_n(\hat{R}_2)$ so that  
$$\tilde{A}_n\cong B_n'+C_n' \; (\mod t\Mat_n(\hat R_0)),$$ 
and thus
$$
t^{n}\tilde{A}_n\cong t^nB_n'+t^{n}C_n' \; (\mod t^{n+1}\Mat_n(\hat R_0)).
$$
So if we define 
\begin{align*}
B_n & = B_{n-1} +  t^{n}B_n'\\
C_n&=C_{n-1}+t^{n}C_n',
\end{align*}
then 
\begin{alignat*}{2}
A&= B_{n-1}C_{n-1} + t^{n}\tilde{A}_n \\
&\cong B_{n-1}C_{n-1} + t^{n}B_n'+t^n C_n' &\quad&\; (\mod t^{n+1}\Mat_n(\hat R_0))\\
&\cong (B_{n-1} + t^nB_n')(C_{n-1} + t^nC_n')&& \; (\mod t^{n+1}\Mat_n(\hat R_0))\\
&\cong B_nC_n &&\; (\mod t^{n+1}\Mat_n(\hat R_0)),
\end{alignat*}
where the second to last congruence uses that 
\begin{alignat*}{2}
B_{n-1} \; &\cong B_0  &\; &\cong I \; (\mod t\Mat_n(M_1))\qquad \text{and}\\
C_{n-1} \; &\cong C_0 &\; &\cong I \; (\mod t\Mat_n(\hat R_2)).
\end{alignat*}
This finishes the proof.
\end{proof}

In Proposition~\ref{globalsum} and Lemma~\ref{localsum} it will be shown that the hypothesis of Proposition~\ref{factorizationcloseto1} (i.e.\ the sum decomposition with respect to some module $M_1$) holds in the situations of our main results.

%-------------------------------------------------------------------------------
\subsection{An intersection lemma}
%-------------------------------------------------------------------------------
Let $T$ be a complete domain with $(t)\subset T$ prime (e.g.\ a complete discrete valuation ring $T$ with uniformizer $t$).  
Let $M\subseteq M_1,M_2\subseteq M_0$ be $T$-modules
with $M \cap tM_i = tM$ and $M_i \cap tM_0 = tM_i$. Then $M/tM = M/(M \cap tM_i) \subseteq M_i/tM_i$ for $i=0,1,2$, and similarly $M_i/tM_i \subseteq M_0/tM_0$ for $i=1,2$. Hence we can form the intersection $M_1/tM_1\cap M_2/tM_2$ in $M_0/tM_0$; and this intersection contains $M/tM$.  Under certain additional hypotheses, the next lemma asserts that if this containment is actually an equality then $M_1\cap M_2=M$.

\begin{lemma}\label{generalintersect}
Let $T$ be a complete domain with $(t) \subset T$ prime, and let
$M\subseteq M_1,M_2\subseteq M_0$ be $T$-modules with no $t$-torsion such that 
$M$ is $t$-adically complete, 
with $M \cap tM_i = tM$ and $M_i \cap tM_0 = tM_i$, 
and with $\bigcap_{j=1}^\infty t^jM_0 = (0)$. Assume that $M_1/tM_1\cap M_2/tM_2=M/tM$. Then
$M_1\cap M_2=M$ (where the intersection is taken inside $M_0$).
\end{lemma}

\begin{proof} 
After replacing $M_0$ by its submodule $M_1 + M_2  \subseteq M_0$, we may assume that those two modules are equal.  So by
the intersection hypothesis, we have an exact sequence
$$0 \to M/tM \to M_1/tM_1 \times M_2/tM_2 \to M_0/tM_0 \to 0,$$
via the diagonal inclusion of $M/tM$ and the subtraction map to $M_0/tM_0$.  Let $N =  M_1 \cap M_2$ inside $M_0$.  So $M \subseteq N$, and
$tN = tM_1 \cap tM_2$ because $M_0$ has no $t$-torsion.  Tensoring the exact sequence 
$0 \to N \to M_1 \times M_2 \to M_0 \to 0$ over $T$ with $T/(t)$ yields the exact sequence $N/tN \to M_1/tM_1 \times M_2/tM_2 \to M_0/tM_0 \to 0.$   But the first map in this sequence is injective because 
$tN = tM_1 \cap tM_2$; so 
$$0 \to N/tN \to M_1/tM_1 \times M_2/tM_2 \to M_0/tM_0 \to 0$$
is exact.
Hence the natural map $M/tM \to N/tN$ is an isomorphism; thus $M \cap tN = tM$ and $N = M + tN$.  For all $j \ge 0$, $M \cap t^{j+1}N 
= M \cap tN \cap t^{j+1}N = tM \cap t^{j+1}N = t(M \cap t^jN)$, where the last equality uses that $N$ has no $t$-torsion.  So by induction,  for all $j \ge 0$ we have that $M \cap t^jN = t^jM$, and also that $N = M+t^jN$.   

So for $n \in  M_1 \cap M_2 = N$, there is a sequence of elements $m_j \in M$ with $n-m_j \in t^jN$.  If $h>j$ then $m_h-m_j \in M \cap t^jN = t^jM$.  Since $M$ is $t$-adically complete, there exists an element $m \in M$ and a sequence $i_j \to \infty$ such that $m-m_j \in t^{i_j}M$ for all $j$.  We may assume $i_j \le j$ for all $j$.
Thus $n-m = (n-m_j) - (m-m_j) \in t^{i_j}N \subseteq t^{i_j}M_0$ for all $j$.  But $\bigcap_{j=1}^\infty t^{i_j}M_0 = \bigcap_{j=1}^\infty t^jM_0 = (0)$. So $n-m=0$ and $n=m \in M$.
\end{proof}

%%%%%%%%%%%%%%%%%%%%%%%%%%%%%%%%%%%%%%%%%%%%%%%%%%%%%%%%%%%%%%%%%%%%%%%%%%%%%%%%
\section{The global case}\label{globalsection}
%%Section 4
%%%%%%%%%%%%%%%%%%%%%%%%%%%%%%%%%%%%%%%%%%%%%%%%%%%%%%%%%%%%%%%%%%%%%%%%%%%%%%%%
We now turn to proving our patching result in a global context, in which we consider a smooth projective curve $\hat X$ over a complete discrete valuation ring $T$, and use patches that are obtained from subsets $U_1,U_2$ of the closed fibre $X$ of $\hat X$.  These subsets are permitted to be Zariski open subsets of $X$, but can also be more general.  The strategy is to show that the factorization and intersection conditions of Section~\ref{setup} hold, employing the results of Section~\ref{preliminarysection}.

%-------------------------------------------------------------------------------
\subsection{Factorization}
%-------------------------------------------------------------------------------

In order to apply the results from the last section to patching, we will need to show that
the hypothesis of Proposition~\ref{factorizationcloseto1} is satisfied, i.e., that there is a certain 
additive decomposition.

As before, $T$ is a complete discrete valuation ring with uniformizer $t$.  Let $\hat{X}$ be a projective 
$T$-curve with closed fibre $X$, and let  
$P\in X$ be a closed point at which $\hat X$ is smooth.  A {\bf lift} of $P$ to $\hat X$ is an effective prime divisor $\hat P$ on $\hat X$ whose restriction to $X$ is the divisor $P$.  Such a lift always exists.  Specifically, given $P$, let $\bar{\pi}$ be a uniformizer of the local ring ${\mathcal O}_{X,P}$
and let $\pi \in {\mathcal O}_{\hat{X},P}$ be a lift of $\bar{\pi}$. Then the maximal ideal of 
${\mathcal O}_{\hat{X},P}$ is generated by $\pi$ and $t$, and we may take $\hat P$ to be the connected 
component of the zero locus of $\pi$ that contains $P$. 

More generally, if $D = \sum_{i=1}^r a_iP_i$ is an effective divisor on $X$, and if $\hat P_i$ is a lift of $P_i$ to $\hat X$ as above, we call $\hat D := \sum_{i=1}^r a_i\hat P_i$ a {\bf lift} of $D$ to $\hat X$.

The following two propositions are preliminary technical results, which can be avoided in the special case that $T = k[[t]]$ for some field $k$ and $\hat X = X \times_k k[[t]]$.  (Namely there, if we choose the lift $\hat P = P \times_k k[[t]]$, then the next two propositions hold easily by extending constants from $k$ to $k[[t]]$.  See also \cite{ober} for a discussion of this special case.)

As usual, for a Cartier divisor $D$ on a scheme $Z$, we let $L(Z,D) = \Gamma(Z,{\cal O}_Z(D))$, the set of rational functions on $Z$ whose pole divisor is at most $D$.

\begin{prop}\label{inverselimit}
Let $T$ be a complete discrete valuation ring with uniformizer $t$,
and let $\hat{X}$ be a smooth connected projective $T$-curve with closed fibre $X$ of genus $g$. 
Let $D$ be an effective divisor on $X$ and let $\hat D$ be a lift of $D$ to $\hat X$. Then
\renewcommand{\theenumi}{\alph{enumi}}
\begin{enumerate}
\item $L(\hat{X},\hat{D})$ is a finitely generated $T$-module, and\label{inverselimit1}
\item if $D$ has degree $> 2g-2$, the sequence 
$$0 \to tL(\hat{X},\hat{D}) \to L(\hat{X},\hat{D})\to L(X,D) \to 0$$ 
is exact.\label{inverselimit2}
\end{enumerate}
\end{prop}

\begin{proof}
(\ref{inverselimit1}) Since $\hat X$ is projective over $T$, the $T$-module $L(\hat X, \hat D) = \Gamma(\hat X, \O(\hat D))$ is 
finitely generated (\cite{hartshorne}, II, Theorem~5.19 and Remark~5.19.2); so the first part holds.

(\ref{inverselimit2}) Suppose that $D$ is an effective divisor on $X$ of degree $d > 2g-2$, with a lift $\hat D$ to $\hat X$.  The general fibre $X^\circ$ of $\hat X$ has genus equal to $g$ because the arithmetic genus is constant for a flat family of curves, by 
\cite{hartshorne}, III, Corollary~9.10.  
Also, $\hat D$ is flat over the discrete valuation ring $T$, since it is torsion-free because its support does not contain the closed fibre $X$.  So by the same result in \cite{hartshorne} on constancy of invariants in flat families, the degree $d$ of the closed fibre $D$ of $\hat D$ is equal to the degree of the general fibre $D^\circ$, viewed as a divisor on $X^\circ$.  

Applying the Riemann-Roch Theorem (\cite{serreagcf}, Chapter II, Theorem~3) to the curves $X^\circ$ and $X$, both $L(X^\circ,D^\circ)$ and $L(X,D)$ are vector spaces of dimension $r := d+1-g$ 
over the fraction field $K$ of $T$ and the residue field $k$ of $T$, respectively.
Since $L(\hat X, \hat D)$ is a submodule of the function field $F$ of $\hat X$, it is torsion-free.  
But $T$ is a principal ideal domain and $L(\hat X,\hat D) \otimes_T K = L(X^\circ, D^\circ)$ 
is an $r$-dimensional $K$-vector space; so the finitely generated torsion-free $T$-module $L(\hat X, \hat D)$ 
is free of rank $r$.  
Thus the injection $L(\hat X,\hat D)/tL(\hat{X},\hat D) \rightarrow L(X,D)$ induced by the map 
$L(\hat X,\hat D) \to L(X,D)$ is an isomorphism of $k$-vector spaces, which implies the result.
\end{proof}

\begin{remark}
Alternatively, one could deduce this from Zariski's Theorem on Formal Functions 
(\cite{hartshorne}, III, Theorem 11.1 and Remark 11.1.2).  
But the proof given here is more elementary, and the above assertion will suffice for our purposes.
\end{remark}

Before we proceed, we introduce some notation that will be frequently used in
the sequel.

\begin{notation}\label{globalnotation}
Let $T$ be a complete discrete valuation ring with uniformizer $t$, and let $\hat X$ be a smooth connected
projective $T$-curve with closed fibre $X$ and function field $F$. 
Let $R_\varnothing$ denote the local ring of $\hat X$ at the generic point of $X$.
Given a subset $U$ of $X$, we
introduce the following objects:
\begin{itemize}
\item We set $R_U:=\{f \in R_\varnothing|\; \text{$f$ is regular on $U$}\}$, and
  we let $\hat{R}_U$ denote the $t$-adic completion of $R_U$.\label{not1}
\item If $U\neq X$, then $F_U$ denotes the fraction
  field of $\hat{R}_U$, and we set $\hat U:=\Spec\hat{R}_U$.\label{not2}
  If $U=X$, then $F_U:=F$.
\end{itemize}
In particular, $\hat R_\varnothing$ is the completion of the local ring of $\hat X$ at the generic point of the closed fibre $X$; this is a complete discrete valuation ring with uniformizer $t$, having as residue field the function field of $X$. 
Also, $F\leq F_U$ for all $U$, and $F_U\leq F_V$ if $V\subseteq U$.  (As we will see in Corollary~\ref{weierstrassvar} below, for {\it any} $U \subseteq X$, the field $F_U$ is the compositum of its subrings $F$ and $\hat R_U$.)

\end{notation}

The next result is an analog of Proposition~\ref{inverselimit} for subsets $U$ of the closed fibre $X$.  If $D$ is an effective divisor on $X$ that is supported on $U$, then we may regard a lift $\hat D$ of $D$ to $\hat X$ as a divisor on the scheme $\hat U$; and so as above we may consider $L(\hat U,\hat D)$, the rational functions on $\hat U$ with pole divisor at most $\hat D$.  We similarly write $L(U,D):=\{f \in k(X)|\; ((f)+D)|_U\geq 0\}$.

\begin{prop}\label{fingen} 
Let $T$ be a complete discrete valuation ring with uniformizer $t$, and let
$\hat X$ be a smooth connected projective $T$-curve with closed fibre $X$ of genus $g$. 
Let $U$ be a proper subset of $X$, let $D$ be an effective divisor on $U$, and  
let $\hat D$ be a lift of $D$ to $\hat X$.  Then 
\renewcommand{\theenumi}{\alph{enumi}}
\begin{enumerate}
\item $L(\hat U,\hat{D})$ is a finitely generated $\hat R_{U}$-module, and \label{fingen1}
\item if $D$ has degree greater than $2g-2$, the sequence 
$$0 \to tL(\hat U,\hat{D}) \to L(\hat U,\hat{D})\to L(U,D) \to 0$$
is exact.\label{fingen2}
\end{enumerate}
\end{prop}

\begin{proof}
(\ref{fingen1}) We proceed by induction on $N = {\rm deg}(D)$.  The result holds for $N = 0$ since then $D=0$ and $L(\hat U,\hat{D}) = \hat R_{U}$.  So take $D$ of degree $N>0$ and assume that the result holds for smaller degrees.  Write $D = \sum_{i=1}^r a_i P_i$ for integers $a_i > 0$ and distinct closed points $P_i$ on $U$; and write $\hat D = \sum_{i=1}^r a_i \hat P_i$ for some lifts $\hat P_i$ of the points $P_i$.
Let $D' = D-P_1$ and $\hat D' = \hat D - \hat P_1$.  Then $D'$ is effective of degree less than $N$, and so $L(\hat U,\hat{D}')$ has a finite generating set $S'$ over $\hat R_U$.

Pick a closed point $Q \in X \smallsetminus U$ and a lift $\hat Q$ of $Q$ to $\hat X$.  
Let $X^\circ$ be the generic fibre of $\hat X$, 
and let $P_1^\circ$ and $Q^\circ$ be the generic points of $\hat P_1$ and $\hat Q$.  
The residue field $K_1$ of $\hat R_U$ at $P_1^\circ$ is a finite extension of the fraction field $K$ of $T$, 
and the local ring ${\cal O}_{\hat U,P_1^\circ}$ at $P_1^\circ$ is an equal characteristic discrete valuation ring with residue field $K_1$ (which can also be regarded as the constant subfield of ${\cal O}_{\hat U,P_1^\circ}$).

By the Strong Approximation Theorem (\cite{fieldarith}, Proposition~3.3.1), there is a rational function $f$ on $X^\circ$
(or equivalently, on $\hat X$) that has a pole of order $a_1$ at $P_1^\circ$ and is regular on 
$X^\circ  \smallsetminus \{P_1^\circ, Q^\circ\}$.  After multiplying $f$ by an appropriate power of $t$, we may assume that $f$ is a unit at the generic point of $X$.  
So $f \in L(\hat U, a_1\hat P_1) \smallsetminus L(\hat U, (a_1-1)\hat P_1)$.

For any $h \in L(\hat U, \hat D)$, there is an element
$\bar c \in K_1 \le {\cal O}_{\hat U,P_1^\circ}$ such that $h-\bar cf \in {\cal O}_{\hat U,P_1^\circ}$ has a pole at $P_1^\circ$ of order 
at most $a_1-1$ (since the local ring at $P_1^\circ$ is an equal characteristic discrete valuation ring with 
constant field $K_1$, and since $f$ has a pole of order $a_1$ at $P_1^\circ$).  
Now viewing $K_1$ as the residue field of ${\cal O}_{\hat U,P_1^\circ}$, lift $\bar c \in K_1$ to $c \in \hat R_U$.  Then $h - cf \in L(\hat U, \hat D')$, and hence it is an $\hat R_U$-linear combination of the elements of $S'$.  So $S := S' \cup \{f\}$ is a finite $\hat R_U$-generating set for $L(\hat U, \hat D)$.

(\ref{fingen2}) The kernel of $L(\hat U,\hat{D})\to L(U,D)$ is clearly $tL(\hat U,\hat{D})$.  To show surjectivity of
$L(\hat U,\hat{D})\to L(U,D)$, 
let $\bar b \in L(U,D)$ and consider $\bar{b}$ 
as a rational function on $X$.  Let $\{P_i'\;|\; i=1,\ldots,m\}$ be the set of poles of $\bar{b}$ that are not 
in $U$, of orders $n_i\in \N$. Then $\bar{b}\in L(X,\sum\limits_{i=1}^m n_iP_i'+D)$.  If $D$ has degree greater than $2g-2$, then so does $\sum\limits_{i=1}^m n_iP_i'+D$.
For each $i$ choose a lift $\hat P_i'$ of $P_i'$ to $\hat X$.
Proposition~\ref{inverselimit}(\ref{inverselimit2}) then gives a preimage $b$ of $\bar{b}$ in 
$L(\hat X,\sum\limits_{i=1}^m n_i\hat{P}_i'+\hat{D})\subseteq L(\hat{U},\hat{D})$, as desired.
\end{proof}

\begin{prop}\label{globalsum}
Let $T$ be a complete discrete valuation ring with uniformizer $t$. 
Let $\hat{X}$ be a smooth connected projective $T$-curve with function field $F$ and closed fibre $X$ of genus $g$.
Consider proper subsets $U_1,U_2\subset X$, with $U_0:=U_1\cap U_2$ empty. 
Let $\hat{R}_i:=\hat{R}_{U_i}$, $F_i:=F_{U_i}$.
Then there exists a finite $\hat{R}_1$-submodule $M_1$ of $\hat R_0 \cap F_1 \subseteq F_0$ with the following property:
For every $a\in \hat{R}_0$ there exist $b\in M_1$ and $c\in \hat{R}_2$ so that $a\cong b+c \; (\mod t\hat R_0)$.
More precisely, for any closed point $P \in U_1\subset X$, for any lift $\hat P$ of $P$ to $\hat X$, and for any non-negative integer 
$N > 2g-2$,
the module $M_1$ can be chosen as $L(\hat{U}_1,N\hat{P})$.
\end{prop}

\begin{proof}
Let $P$ and $\hat P$ be as above, and let $M_1 = L(\hat{U}_1,N\hat{P})$ for some non-negative integer $N > 2g-2$.  Thus $M_1 \subseteq \hat R_0 \cap F_1$.
By Proposition~\ref{fingen}(\ref{fingen1}), the $\hat{R}_1$-module $M_1$ is finitely generated.

Given $a\in \hat{R}_0$, its mod $t$ reduction $\bar a \in \hat{R}_0/t\hat{R}_0$ may be viewed as a rational
function on $X$.  Consider the family of rational functions $\{f_Q\}_{Q \in X}$ on $X$ given by $f_Q = \bar a$ for $Q \in U_1$ and $f_Q = 0$ for $Q \notin U_1$.  Since $N > 2g-2$, the Riemann-Roch Theorem (\cite{serreagcf}, Chapter II, Theorem~3)
implies that $H^1(X,{\cal O}_X(NP))=0$.  Hence by \cite{serreagcf}, Chapter II, Proposition~3, there is a rational function $\bar c$ on $X$ such that  
$f_Q-\bar c$ is regular at $Q$ for all $Q \ne P$, and such that $f_P - \bar c$ has a pole at $P$ of order at most $N$.  In particular, $\bar c$ is regular on $U_2$.  Thus $\bar c \in \hat R_2/t\hat R_2$, and $\bar c$ is the reduction of some $c \in \hat R_2$.  By the definition of $\bar c$, the rational function $\bar b := \bar a - \bar c$ on $X$ is regular on $U_1$ except possibly at $P$, where it has a pole of order at most $N$; i.e., $\bar b \in L(U_1,NP)$. 
Proposition~\ref{fingen}(\ref{fingen2}) implies that $\bar b$ is the image of an element $b \in L(\hat{U}_1,N\hat{P}) = M_1$, since  $N > 2g-2$; and then $a\cong b+c \; (\mod t\hat R_0)$.
\end{proof}

The main result of this section is a factorization result for use in patching.  As above, we use Notation~\ref{globalnotation}.

\begin{thm}\label{globalfactorization}
Let $T$ be a complete discrete valuation ring, and let  
$\hat X$ be a smooth connected projective $T$-curve with  closed fibre $X$. 
Let $U_1$, $U_2$ be subsets of $X$ and assume that $U_0:=U_1 \cap U_2$ is empty.
Let $F_i:=F_{U_i}$  and $\hat R_i = \hat R_{U_i}$ ($i=0,1,2$). 
Then for every matrix $A\in \Gl_n(F_0)$ there exist
matrices $A_1\in \Gl_n(F_1)$ and $A_2 \in \GL_n(F_2)$ such that $A=A_1A_2$.
\end{thm}

\begin{proof} 
We may assume that $U_1, U_2$ are proper subsets of $X$; otherwise the assertion is trivial.  As observed at Notation~\ref{globalnotation},
$\hat R_0 = \hat R_\varnothing$ is a complete discrete valuation ring whose residue field $\hat R_0/t\hat R_0$ is the function field of $X$ (which is also the fraction field of $\hat R_1/t\hat R_1$).  Moreover the uniformizer $t$ of $T$ is also a uniformizer for $\hat R_0$.
By Proposition~\ref{globalsum}, there exists a finite $\hat{R}_1$-module $M_1 \subset \hat R_0 \cap F_1$
satisfying the hypothesis of Proposition~\ref{factorizationcloseto1}.  So by Proposition~\ref{factorizationcloseto1}, for every $A \in \GL_n(\hat R_0)$ that is congruent to the identity modulo $t$, there exist $A_1 \in \GL_n(F_1)$ and $A_2 \in \GL_n(\hat R_2)$ such that $A = A_1A_2$.  By Lemma~\ref{conditionalfactorization}, the same conclusion then holds for {\it any} matrix $A \in \Mat_n(\hat R_0)$ having non-zero determinant.  Finally, for any $A\in \Gl_n(F_0)$, there is an $r \ge 0$ such that $t^rA \in \Mat(\hat R_0)$  with non-zero determinant.  Since $t^rI \in \Gl_n(F_1)$, the conclusion again follows.
\end{proof}

The above proof actually shows a stronger result: Namely, every matrix $A\in \Gl_n(F_0)$ may be factored as $A=A_1A_2$, for some matrices $A_1\in \Gl_n(F_1)$ and $A_2 \in \GL_n(\hat R_2)$.

A generalization of Theorem~\ref{globalfactorization} in which $U_1 \cap U_2$ can be non-empty appears in Theorem~\ref{globalfactgeneral} below.

%-------------------------------------------------------------------------------
\subsection{Intersection}
%-------------------------------------------------------------------------------

We continue to use Notation~\ref{globalnotation}.  

\begin{prop}[Weierstrass Preparation]\label{weierstrass}
Let $T$ be a complete discrete valuation ring and let $\hat{X}$ be a smooth connected projective $T$-curve with function field $F$ and 
closed fibre $X$.  Suppose that $U\subseteq X$.
Then every element $f\in \hat R_U$ may be written as $f=bu$ with $b\in F$ and $u\in
\hat R_U^\times$.
\end{prop}

\begin{proof}
If $U = X$ then $\hat R_U = T \subset F$, and the result is immediate.
If $U = \varnothing$, then $\hat R_U$ is a discrete valuation ring, and the uniformizer $t$ of $T$ is also a uniformizer of $\hat R_U$.  In this case the result also follows easily, by taking $b$ to be a power of $t$.  So from now on we assume that $U \ne X, \varnothing$.

Let $U_1:=X\smallsetminus U$.  Thus $U_1\cap U = \varnothing$ and $\hat R_\varnothing/t\hat R_\varnothing$ is the function field of $X$.  Let $f \in \hat R_U$; we may assume $f \ne 0$ since otherwise the result is trivial.  Since $t \in F$, after factoring out a power of $t$ from $f$, we may assume that $f \notin t\hat R_U$.  Let $\bar f \in \hat R_U/t\hat R_U \subset \hat R_\varnothing/t\hat R_\varnothing$ be the reduction of $f$ modulo $t$.  Here $\bar f$ is a non-zero rational function on $X$ whose pole divisor $D$ is supported on $U_1$.  If $D=0$, then $\bar f$ is a nonzero constant function on $X$, and hence is a unit in $\hat R_U/t\hat R_U$, the ring of functions on $U$.  In this case $f$ is a unit in the $t$-adically complete ring $\hat R_U$, and we may take $u=f$, $b=1$. 

So now assume instead that $D$ is a nonzero effective divisor, hence of degree at least $1$.
Let $\hat D$ be a lift of $D$ to $\hat X$, and pick a positive integer $N > 2g-2$, where $g$ is the genus of $X$.  Thus $\bar f \in L(X,D) \subseteq L(X,ND)$.
By Proposition~\ref{inverselimit}(\ref{inverselimit2}), there exists some $\hat f \in L(\hat X,N\hat D) \subset R_U \subseteq F \cap \hat R_U$ whose reduction mod $t\hat R_U$ is $\bar f$; thus $f \equiv \hat f \; (\mod t\hat R_U)$.  Here $\hat f \not\in t\hat R_U$ because $\bar f \ne 0$; so $\hat f$ is invertible in the $t$-adically complete ring $\hat R_\varnothing$.
Let $\tilde f=f/\hat f \in \hat R_\varnothing$.  Hence $\tilde f \equiv 1 \; (\mod t\hat R_\varnothing)$.  Let $P$ be a point of $U_1$ and let $\hat P$ be a lift of $P$ to $\hat X$.  By Proposition~\ref{globalsum} (with $U_2=U$ and $U_0 = \varnothing$), Proposition~\ref{factorizationcloseto1} allows us to write $\tilde f=f_1f_2$ with $f_1\in L(\hat U_1,N\hat P)$ and $f_2\in \hat{R}_U^\times$. 
So $\hat f f_1 = ff_2^{-1}\in L\cap \hat{R}_U$, where $L := L(\hat U_1,N\hat D + N\hat P)$, using that $\hat f \in L(\hat X,N\hat D) \subseteq L(\hat U_1,N\hat D)$. 
By Proposition~\ref{fingen}(\ref{fingen2}), $L/tL = L(U_1,ND + NP)$. 
Thus $L/tL\cap \hat{R}_U/t\hat{R}_U= L(X,ND+NP) = L(\hat X,N\hat D + N\hat P)/tL(\hat X,N\hat D +N\hat P)$ by Proposition~\ref{inverselimit}(\ref{inverselimit2}).  Applying Lemma~\ref{generalintersect} 
to the four $T$-modules $L(\hat X,N\hat D+N\hat P) \subseteq L, \hat R_U \subseteq \hat R_\varnothing$ yields that $\hat f f_1\in L\cap \hat{R}_U = L(\hat X,N\hat D+N\hat P) \subset F$.  Hence we may take $b=\hat f f_1 \in F$ and $u=f_2 \in \hat R_U^\times$.
\end{proof}

Note that if $\hat X = \P^1_T$ and $U$ consists of a single point, then this assertion is related to the classical form of the Weierstrass preparation theorem (e.g.\ see \cite{grifharris}, p.8).

\begin{cor}\label{weierstrassvar}
With notation as in Proposition~\ref{weierstrass}, every element $f$ in the fraction field of $\hat R_U$ may be written as $f=bu$ with $b\in F$ and
$u\in \hat{R}_U^\times$.  Hence $F_U$ is the compositum of $\hat R_U$ and $F$.
\end{cor}

Here the first assertion is immediate from the above proposition, and the second assertion then follows from the definition of $F_U$ in Notation~\ref{globalnotation}, using $\hat R_X = T$.

We are now in a position to prove the intersection result needed for patching.

\begin{thm}\label{globalintersection}
Let $T$ be a complete discrete valuation ring, let $\hat X$ be a 
smooth connected projective $T$-curve with closed fibre $X$.
Let $U_1$, $U_2$ be subsets of $X$, and write $U=U_1 \cup U_2$,    
$U_0 = U_1 \cap U_2$. 
Then $F_{U_1} \cap F_{U_2} = F_U$ inside $F_{U_0}$.
\end{thm}

\begin{proof} Let $\hat{R}_i:=\hat{R}_{U_i}$ and $F_i := F_{U_i}$, and let $t$ be a uniformizer of $T$.
We need only show that $F_1\cap F_2\subseteq F_U$, the reverse inclusion being trivial. Take an element $f \in F_1 \cap F_2$.  By Corollary~\ref{weierstrassvar}, $f=f_1u_1=f_2u_2$ 
with $f_i \in F \leq F_U$ and $u_i\in \hat R_i^\times$.  We wish to show that $f \in F$.

First, assume that $U\neq X$.  Thus $F_U$ is the fraction field of $\hat R_U$.
Write $f_i=a_i/b_i$ with
$a_i,b_i\in\hat R_U$.  So $f=a_1u_1/b_1=a_2u_2/b_2$.  Hence $a_1b_2u_1=a_2b_1u_2$, where
the left side is in $\hat R_1$ and the right side is in $\hat R_2$.  
Since $\hat R_1/t\hat R_1 \cap \hat R_2/t\hat R_2 = \hat R_U/t\hat R_U$, the
hypotheses of Lemma~\ref{generalintersect} are seen to hold in this situation (with $M_i := \hat R_i$, $M := \hat R_U$); so 
$\hat R_1
\cap \hat R_2 = \hat R_U$  and
$a_1b_2u_1\in \hat R_U$.  
But then 
$f=a_1u_1/b_1=a_1b_2u_1/b_1b_2$, where the numerator and denominator are both in
$\hat R_U$; i.e., $f\in F_U$.

Next suppose that $U=X$, so that $F_U = F$.  We may assume that $U_1, U_2$ are proper subsets of $X$, since otherwise the assertion is trivial.  So $U_2$ is not contained in $U_1$.  Pick a closed point $P \in U_2 \smallsetminus U_1$ and a lift $\hat P \in \hat U_2 \subset \hat X$.
By Propositions~\ref{inverselimit}(\ref{inverselimit2}) and \ref{fingen}(\ref{fingen2}), the mod $t$ reduction maps
$L(\hat{X},N\hat{P})\rightarrow L(X,NP)$ and
$L(\hat{U}_2,N\hat{P})\rightarrow L(U_2,NP)$ are surjective for $N$ sufficiently large. 
Then, by Lemma~\ref{generalintersect}, $\hat R_1 \cap L(\hat U_2, N\hat P) = L(\hat X, N\hat P)$ for $N \gg 0$, 
using in particular that the same statement is true modulo $t$. 
Let $R' = \bigcup_{N=0}^\infty L(\hat X, N\hat P)$, the ring of regular functions on 
$\hat{X}\smallsetminus \hat{P}$; and let 
$\hat R_2' = \bigcup_{N=0}^\infty L(\hat U_2, N\hat P)$, 
the ring of regular functions on 
$\hat U_2 \smallsetminus \hat{P}$.  The above intersection for $N \gg 0$ implies that 
$\hat{R}_1\cap \hat R_2'=R'$, and in particular that 
$R' \subset \hat R_1$.  Also, 
$\hat R_2 \subset \hat R_2'$, and 
$F$ is the fraction field of $R'$.  Proceeding as in the previous paragraph but with $\hat R_U$ and $\hat R_2$ respectively replaced by $R'$ and $\hat R_2'$, we may write $f_i=a_i/b_i$ with $a_i,b_i\in R'$.  Thus $a_1b_2u_1=a_2b_1u_2 \in \hat R_1 \cap \hat R_2' = R'$
and so $f = a_1b_2u_1/b_1b_2 \in F$.
\end{proof}

Using Theorem~\ref{globalintersection}, we next obtain a strengthening of the factorization result Theorem~\ref{globalfactorization} that applies to more general pairs $U_1,U_2$.  This result, which may be regarded as a form of Cartan's Lemma (\cite{cartan}, Section~4, Th\'eor\`eme~I), also generalizes Corollary 4.5 of \cite{haranjarden} (which dealt just with the case that the $U_i$ are Zariski open subsets of the line in order to make use of unique factorization of the corresponding rings).

\begin{thm} \label{globalfactgeneral}
Let $T$ be a complete discrete valuation ring, let  
$\hat X$ be a smooth connected projective $T$-curve with closed fibre $X$. 
Let $U_1, U_2 \subseteq X$, let $U_0=U_1 \cap U_2$, and 
let $F_i:=F_{U_i}$ ($i=0,1,2$) under Notation~\ref{globalnotation}. 
Then for every matrix $A\in \Gl_n(F_0)$ there exist
matrices $A_i\in \Gl_n(F_i)$ such that $A=A_1A_2$.
\end{thm}

\begin{proof}
Let $U_2' =  U_2 \smallsetminus U_0$, and write $F_2' = F_{U_2'}$ and $F_0' = F_\varnothing$.  Any $A \in \GL_n(F_0)$ lies in $\GL_n(F_0')$, and so by Theorem~\ref{globalfactorization} we may write $A = A_1A_2$ with $A_1 \in \GL_n(F_1) \leq \GL_n(F_0)$ and $A_2 \in \GL_n(F_2')$.  But also $A_2 = A_1^{-1}A \in \GL_n(F_0)$; and $F_2' \cap F_0 = F_2$ by Theorem~\ref{globalintersection} since $U_2' \cup U_0 = U_2$.  So actually $A_2 \in \GL_n(F_2)$.
\end{proof}

\begin{remark} \label{nonzero} 
In Theorem~\ref{globalfactgeneral}, we cannot replace $\GL_n$ everywhere by $\Mat_n$, as the following example shows.  With notation as above, assume $U_0 \ne U_1, U_2$, and consider the matrix 
$$A = \begin{pmatrix} 1 & a_1 \\a_2 & a_1a_2\end{pmatrix} \in \Mat_n(F_0)$$
with $a_i \in F_i \smallsetminus F_U$.  If there is a factorization $A=A_1A_2$ with $A_i \in \Mat_n(F_i)$, then either $A_1$ or $A_2$ has determinant $0$, since $\det(A)=0$.  Without loss of generality, we may assume $\det(A_1) = 0$ (since otherwise we can interchange the roles of $U_1, U_2$ and consider the transpose of $A$).  So there exist $r,s \in F_1$, not both zero, such that $r(A_1)_1=s(A_1)_2$, where $(A_1)_i$ denotes the $i$th row of $A_1$.  Multiplying by $A_2$ on the right then gives the equality $r(A)_1=s(A)_2$ for the rows of $A$; in particular, $r=sa_2$.  If $s \ne 0$, then $a_2 = r/s \in F_1$.  By assumption, $a_2 \in F_2$, and thus $a_2 \in F_U$ by Theorem~\ref{globalintersection}; a contradiction.  Consequently, $s=0$, and thus $r=sa_2=0$, contradicting the fact that $r,s$ are not both zero.  Hence no such factorization can exist.   
\end{remark}

%-------------------------------------------------------------------------------
\subsection{Patching}
%-------------------------------------------------------------------------------

We now turn to our global patching result for function fields.
We consider an irreducible projective $T$-curve $\hat X$ with closed fibre $X$.  
For any subset $U \subseteq X$ we write $\V(U)$ for $\Vect(F_U)$, where $F_U$ is as in Notation~\ref{globalnotation}.

\begin{thm}\label{globalpatching}
 Let $T$ be a complete discrete
valuation ring and let  
$\hat X$ be a smooth connected projective $T$-curve with closed fibre $X$.
Let $U_1$, $U_2$ be subsets of $X$. 
Then the base change functor
$$\V(U_1 \cup U_2) \to  \V(U_1)\times_{\V(U_1\cap U_2)} \V(U_2)$$
is an equivalence of categories.
\end{thm}

\begin{proof}
In view of Proposition~\ref{harb}, 
the result follows from
the factorization result
Theorem~\ref{globalfactgeneral} and the intersection result
Theorem~\ref{globalintersection}.
\end{proof}

By Proposition~\ref{harb}, the inverse of the above equivalence of categories (up to isomorphism) is given by taking the fibre product of vector spaces.  

\begin{remark}
Theorem~\ref{globalpatching} can also be deduced just from Theorem~\ref{globalfactorization} and Theorem~\ref{globalintersection}, without using Theorem~\ref{globalfactgeneral}.  Namely, the case that $U_0=\varnothing$ follows with Theorem~\ref{globalfactorization} replacing 
Theorem~\ref{globalfactgeneral} in the above proof; and the general case then follows from that by setting $U_2' = U_2 \smallsetminus U_0$ and using the 
equivalences of categories
\begin{equation*}
\begin{split}
\V(U_1) \times_{\V(U_0)} \V(U_2) &= \V(U_1) \times_{\V(U_0)} 
(\V(U_0) \times_{\V(\varnothing) }\V(U_2'))\\ &=
  \V(U_1) \times_{\V(\varnothing)} \V(U_2') = \V(U_1
    \cup U_2') = \V(U_1 \cup U_2).
\end{split}
\end{equation*}
\end{remark}

Theorem~\ref{globalpatching} generalizes to a version that allows patching more than two vector spaces at the same time.
This next result will become important in later applications, where sometimes $U_0$ is empty.

\begin{thm}\label{globalpatchingseveral}
Let $T$ be a complete discrete
valuation ring and let  
$\hat X$ be a smooth connected projective $T$-curve with closed fibre $X$.
Let $U_1,\ldots,U_r$ denote subsets of $X$,  and assume that the pairwise intersections $U_i\cap U_j$ (for $i \ne j$) are all equal to a common subset $U_0 \subseteq X$. Let $U=\bigcup \limits_{i=1}^r U_i$. 
Then the base change functor
$$\V(U)\rightarrow \V(U_1) \times_{\V(U_0)} \cdots  \times_{\V(U_0)} \V(U_r)$$
is an equivalence of categories.
\end{thm}

\begin{proof}
We proceed by induction; the case $r=1$ is trivial. Since
$$\biggl(\,\bigcup\limits_{i=1}^{r-1} U_i\biggr)\cap U_r=\bigcup\limits_{i=1}^{r-1}
(U_i\cap U_r)=U_0,$$ Theorem~\ref{globalpatching} yields an equivalence of
categories
$$\V\biggl(\,\bigcup\limits_{i=1}^r
  U_i\biggr)=\V\biggl(\,\bigcup\limits_{i=1}^{r-1} U_i\biggr)\times_{\V(U_0)}\V(U_r).$$
By the inductive hypothesis, the first factor on the right hand side is equivalent to the category $\V(U_1) \times_{\V(U_0)} \cdots  \times_{\V(U_0)} \V(U_{r-1})$, proving the result.
\end{proof}

Note that by Theorem~\ref{globalintersection} and induction, 
$F_U$ is the intersection of the fields $F_{U_1},\dots,F_{U_r}$ inside $F_{U_0}$.  So as
with Theorem~\ref{globalpatching}, the inverse to the equivalence of categories (up to isomorphism) in Theorem~\ref{globalpatchingseveral} is given by taking the fibre product of the given $F_{U_i}$-vector spaces ($i=1,\dots,r$) over the given  $F_{U_0}$-vector space; this is by Proposition~\ref{patchprob}.

%%%%%%%%%%%%%%%%%%%%%%%%%%%%%%%%%%%%%%%%%%%%%%%%%%%%%%%%%%%%%%%%%%%%%%%%%%%%%%%%
\section{The Complete Local Case}\label{localsection}
%%Section 5
%%%%%%%%%%%%%%%%%%%%%%%%%%%%%%%%%%%%%%%%%%%%%%%%%%%%%%%%%%%%%%%%%%%%%%%%%%%%%%%%
In this section, we will prove a different patching result, 
in which complete local rings are used at one or more points, and which is related to results in \cite{FPABP}, Section~1. The proof here relies on the case
dealt with in Section~\ref{globalsection}. Again, the ingredients we need are a
factorization result and an intersection result. We use the following

\begin{notation}\label{localnotation}
Let $\hat R$ be a $2$-dimensional regular local domain with maximal ideal $\m$
and local parameters $f,t$, such that $\hat R$ is $t$-adically complete. 
Let $\hat R_1$ be the $\m$-adic completion of $\hat R$, let $\hat R_2$ be the
$t$-adic completion of $\hat R[f^{-1}]$, and let $\hat R_0$ be the $t$-adic
completion of $\hat R_1[f^{-1}]$.  
Also let $\bar R := \hat R/t\hat R$ and
let $\bar R_i = \hat R_i/t\hat R_i$ for $i=0,1,2$.
\end{notation}

\begin{lemma}\label{notationlemma}
In the context of Notation~\ref{localnotation}, the following hold:
\renewcommand{\theenumi}{\alph{enumi}}
\begin{enumerate}
\item $\hat R_1$ is the $f$-adic completion of $\hat R$, and 
$\hat R \le \hat R_i \le \hat R_0$ for $i=1,2$.\label{notlema}
\item $t\hat R_i \cap \hat R = t\hat R$ for $i=0,1,2$, and
$t\hat R_0 \cap \hat R_i = t\hat R_i$ for $i=1,2$.\label{notlemb}
\item $\hat R_2$ and $\hat R_0$ are complete discrete valuation rings with
uniformizer $t$; and
$\bar R$ and $\bar R_2$ are discrete valuation rings with
uniformizer $\bar f$, the mod $t$ reduction of $f$.\label{notlemc}
\item $\bar R_1$ is
the $\bar f$-adic completion of $\bar R$; while $\bar R_2$ and $\bar R_0$ are respectively isomorphic to 
$\bar R[\bar f^{-1}]$ and $\bar R_1[\bar f^{-1}]$, the fraction fields of $\bar R$ and $\bar R_1$.\label{notlemd}
\item $\bar R \le \bar R_i \le \bar R_0$ for $i=1,2$, with $\bar R_1 \cap \bar R_2 = \bar R$ inside $\bar R_0$.\label{notleme}
\end{enumerate}
\end{lemma}

\begin{proof}
(\ref{notlema}) Since $\hat R$ is $t$-adically complete, $\hat R = \displaystyle \lim_\leftarrow \hat R/t^j\hat R$.  In $\hat R$, $(f^{2n},t^{2n}) \subset \m^{2n} \subset (f^n,t^n)$ for all $n \ge 0$.  So the $f$-adic completion of $\hat R$ is $\displaystyle \lim_\leftarrow \hat R/f^i\hat R = 
\lim_\leftarrow  \lim_\leftarrow (\hat R/t^j\hat R)/f^i(\hat R/t^j\hat R) = 
\lim_\leftarrow  \lim_\leftarrow \hat R/(f^i,t^j) =  \lim_\leftarrow \hat R/\m^n = \hat R_1$.  This proves the first part of (\ref{notlema}).

According to \cite{bourbaki}, III, \S3.2, Corollary to Proposition~5, given an ideal $I$ in a commutative ring $A$, the intersection $\bigcap I^n$ is equal to $(0)$ if no element of $1+I$ is a zero-divisor.  Hence the completion maps $\hat R \to \hat R_1$, 
$\hat R[f^{-1}] \to \hat R_2$, and $\hat R_1[f^{-1}] \to \hat R_0$ are injections.  Thus so are $\hat R \to \hat R_2$ and $\hat R_1 \to \hat R_0$.  

It remains to show that $\hat R_2 \to \hat R_0$ is injective.
Since the image $\bar f$ of $f$ is in the maximal ideal of $\hat R/t^j\hat R$, no element of $1 + (\bar f) \subseteq \hat R/t^j\hat R$ is a zero-divisor.  The above result in \cite{bourbaki} then implies that 
$\bigcap_{n=1}^\infty (\bar f^n) = (0)\subset \hat R/t^j\hat R$ for $j \ge 1$.  Hence 
$\bigcap_{n=1}^\infty (t^j,f^n) = (t^j)\subset \hat R$ for each $j$.
Meanwhile, since $\hat R_1$ is the $f$-adic completion of $\hat R$ (as shown above), $\hat R \cap f^n\hat R_1 = f^n\hat R$, and $t^j\hat R$ is $f$-adically dense in $t^j\hat R_1$.  By this density, if $g \in t^j\hat R_1 \cap \hat R$, then $g=t^jr + f^nr_1$ for some $r \in \hat R$ and $r_1 \in \hat R_1$.  But then $f^nr_1 \in \hat R \cap f^n\hat R_1 = f^n\hat R$; i.e.\ $r_1 \in \hat R$.  This shows that $g$ lies in the ideal $(t^j,f^n) \subset \hat R$.  Since this holds for all $n$, 
and since $\bigcap_{n=1}^\infty (t^j,f^n) = t^j\hat R$, it follows that $g \in t^j\hat R$.  That is, $t^j\hat R_1 \cap \hat R \subseteq t^j\hat R$.  The reverse containment is trivial, and so $t^j\hat R_1 \cap \hat R = t^j\hat R$. Hence $t^j\hat R_1[f^{-1}] \cap \hat R[f^{-1}] = t^j\hat R[f^{-1}]$ for all $j \ge 1$; thus the map $\hat R_2 \to \hat R_0$ on completions is injective.

(\ref{notlemb})  It was shown in the proof of part (\ref{notlema}) that $t\hat R_1 \cap \hat R = t\hat R$.  So if $g \in t\hat R_1[f^{-1}] \cap \hat R$, then $f^ng \in t\hat R_1 \cap \hat R = t\hat R$ for some $n$, and hence $g \in t\hat R$ because $f,t$ are a system of local parameters.  Passing to the $t$-adic completion preserves the $t$-adic metric, and so $t\hat R_0 \cap \hat R \subseteq t\hat R$.  The reverse containment is trivial; hence 
$t\hat R_0 \cap \hat R = t\hat R$.  Since $t\hat R_2 \subseteq t\hat R_0$, we then also have $t\hat R_2 \cap \hat R = t\hat R$.

Since the $t$-adic metric is preserved under $t$-adic completion, to prove that $t\hat R_0 \cap \hat R_1 = t\hat R_1$ it suffices to show that $t\hat R_1[f^{-1}] \cap \hat R_1 = t\hat R_1$.  Say $g$ is in the left hand side.  Then for some $n \ge 1$, $f^ng \in t\hat R_1 \cap f^n\hat R_1 = tf^n\hat R_1$; i.e., $g \in t\hat R_1$.  Thus $t\hat R_1[f^{-1}] \cap \hat R_1 \subseteq t\hat R_1$, and the reverse containment is trivial. Finally, the equality $t\hat R_0 \cap \hat R_2 = t\hat R_2$ follows from the assertion $t\hat R_1[f^{-1}] \cap \hat R[f^{-1}] = t\hat R[f^{-1}]$ shown in the proof of part (\ref{notlema}).

(\ref{notlemc}) Since $f \in \m$, the rings $\hat R[f^{-1}]$ and
$\hat R_1[f^{-1}]$ are regular domains of dimension one; and their $t$-adic completions $\hat R_2$ and $\hat R_0$ are thus complete discrete valuation rings with uniformizer $t$.

Since $f,t$ form a system of local parameters at the maximal ideals of the two-dimensional regular local domains $\hat R$ and $\hat R_1$, the reduction $\bar f$ is a local parameter for the reductions $\bar R$ and $\bar R_1$, which are one-dimensional regular local domains.  That is, $\bar R$ and $\bar R_1$ are discrete valuation rings with uniformizer $\bar f$.

(\ref{notlemd}) The first assertion follows from the fact that $\hat R_1$ is the $f$-adic completion of $\hat R$ (proven in part (\ref{notlema})).  The second assertion follows from part  (\ref{notlemc}) and the definitions of $\hat R_2$ and $\hat R_0$.
 
(\ref{notleme}) These assertions follow from the characterizations of $\bar R, \bar R_1, \bar R_2, \bar R_0$ in part (\ref{notlemd}).  
\end{proof}

%-------------------------------------------------------------------------------
\subsection{Factorization}
%-------------------------------------------------------------------------------

\begin{lemma}\label{localsum}
In the context of Notation~\ref{localnotation},  for every $a \in \hat R_0$ there exist 
$b \in \hat R_1$ and $c \in \hat R_2$ such that $a \cong b+c\; (\mod  t\hat R_0)$.
\end{lemma}

\begin{proof}
We may assume $a \ne 0$. Write $v_{\bar f}$ for the $\bar f$-adic valuation on $\bar{R}_0$. 
Let $\bar a$ be the image of $a$ in $\bar R_0 = \hat R_0/t\hat R_0$.  If $v_{\bar f}(\bar a) \ge 0$, then 
$\bar a \in \bar R_1$; and so there exists $b \in \hat R_1$ such that 
$a \equiv b\; (\mod  t\hat R_0)$.  Taking $c=0$ completes 
the proof in this case.  Alternatively, if $v_{\bar f}(\bar a) = -r < 0$, then $f^r a$ has the 
property that its reduction modulo $t\hat R_0$ lies in $\bar R_1 \subset \bar R_0$, since the $\bar f$-adic valuation of 
this reduction is $0$.  Since $\bar R$ is $\bar f$-adically dense in $\bar R_1$ by Lemma~\ref{notationlemma}(d), there exists 
$\bar d \in \bar R$ such that $\bar d \equiv \bar f^r \bar a\; (\mod \bar f^r\bar R_1)$. 
Let $\bar c = \bar f^{-r}\bar d \in \bar R_2$.  Then $\bar f^r(\bar a - \bar c) = \bar f^r\bar a - \bar d \in \bar f^r\bar R_1$,
and so $\bar a - \bar c$ is equal to some element $\bar b \in \bar R_1$.  
Choosing $b \in \hat R_1$ lying over $\bar b$, and $c \in \hat R_2$ lying over $\bar c$,  
completes the proof.
\end{proof}

\begin{thm}\label{localfactorization} 
In the context of Notation~\ref{localnotation}, let $F_i$ be the fraction field of $\hat R_i$.
Then for every $A \in \GL_n(F_0)$ there exist $A_1 \in \GL_n(F_1)$
 and $A_2 \in \GL_n(F_2)$ such that $A=A_1A_2$.
\end{thm}

\begin{proof}
By Lemma~\ref{notationlemma}(d), $\bar R_0$ is a field. By Lemma~\ref{localsum}, the module $M_1:=\hat{R}_1 \subset \hat R_0$ satisfies the hypothesis of
Proposition~\ref{factorizationcloseto1}.  So in the case of matrices $A \in \GL_n(\hat R_0)$ that are congruent to the identity modulo $t\hat R_0$, the assertion follows from that proposition.  The 
result for an arbitrary matrix $A \in \Mat_n(\hat R_0)$ with non-zero determinant then follows from  
Lemma~\ref{conditionalfactorization} (whose other hypotheses are satisfied, by parts (\ref{notlemc}) and (\ref{notlemd}) of Lemma~\ref{notationlemma}).
Finally, the general case of a matrix $A \in \GL_n(F_0)$ then follows since $t^rA \in \Mat(\hat R_0)$ with non-zero determinant for some $r \ge 0$, and since $t^rI \in \GL_n(F_1)$. 
\end{proof}

%-------------------------------------------------------------------------------
\subsection{Intersection}
%-------------------------------------------------------------------------------
The proof of Weierstrass preparation in the local case does not entirely
parallel the global case; instead, we require the following lemma.

\begin{lemma}\label{localproduct}
In the context of Notation~\ref{localnotation}, every unit $a \in \hat R_0^\times$ may be written as $a=bc$ 
for some units $b \in \hat R_1^\times$ and $c \in \hat R_2^\times$.
\end{lemma}

\begin{proof}
Since $a \in \hat R_0^\times$,
$a \not \equiv 0 \; (\mod t\hat R_0)$.  So the reduction of $a$ modulo $t\hat R_0$ 
is a non-zero element of $\bar R_0 = \bar R_1[\bar f^{-1}]$, 
and hence is of the form $\bar f^s \bar u$ for some integer $s$ and some unit 
$\bar u \in \bar R_1$.  Choose $u \in \hat R_1$ 
with reduction $\bar u$.  Thus $u$ is a unit in the $t$-adically complete ring $\hat R_1$ and $f^s$ is a unit in 
$\hat R_2$.  Replacing $a$ by $u^{-1} a f^{-s}$, we may assume that 
$a \equiv 1 \; (\mod t\hat R_0)$.  

Since $\hat R_1, \hat R_2$ are $t$-adically complete, it now suffices to define sequences of units $b_m \in \hat R_1$, $c_m \in \hat R_2$ such that 
$$b_{m+1} \equiv b_m \; (\mod t^{m+1}\hat R_1), \ c_{m+1} \equiv c_m  \; (\mod t^{m+1}\hat R_2), \ a \equiv b_m c_m \; (\mod t^{m+1}\hat R_0)$$
for all $m \ge 0$.  This will be done inductively.

Take $b_0=1$, $c_0=1$.  Suppose $b_{m-1}$ and $c_{m-1}$ have been defined, with $m \ge 1$.  
Thus $b_{m-1} \equiv 1 \; (\mod t\hat R_1)$, $c_{m-1} \equiv 1 \; (\mod t\hat R_2)$, and $d_m := a b_{m-1}^{-1} - c_{m-1}$ is divisible by 
$t^m$ in $\hat R_0$.
So $d_m=\delta_m t^m$ for some $\delta_m \in \hat R_0$; denote its reduction  modulo $t\hat R_0$ by
$\bar \delta_m \in \bar R_0$.  For some non-negative integer $i$ we have $\bar f^i\bar\delta_m \in \bar R_1$.  But $\bar R$ is $\bar f$-adically dense in $\bar R_1$; so there 
exists $\bar\varepsilon_{m} \in \bar R$ such that  
$\bar\varepsilon_m \equiv \bar f^i\bar\delta_m \; (\mod \bar f^i\bar R_1)$.  So 
$\bar b_m' := \bar\delta_m - \bar f^{-i} \bar\varepsilon_m \in \bar R_1$
and
$\bar c_m' := \bar f^{-i} \bar\varepsilon_m \in \bar R[\bar{f}^{-1}] = \bar R_2$.  Choose elements 
$b_m' \in \hat R_1$ and $c_m' \in \hat R_2$ respectively lying over $\bar b_m' \in \bar R_1$ 
and $\bar c_m' \in \hat R_2$, and let $b_m = b_{m-1} + b_m't^m \in \hat R_1$ and
$c_m = c_{m-1} + c_m't^m \in \hat R_2$.  Thus $b_m \equiv b_{m-1} \; (\mod t^m\hat R_1)$, \ $c_m \equiv c_{m-1} \; (\mod t^m\hat R_2)$, and 
$a b_{m-1}^{-1} - c_{m-1} = d_m = \delta_m t^m \equiv b_m't^m + c_m't^m \; (\mod t^{m+1}\hat R_0)$. 
So $a \equiv b_{m-1}c_{m-1}+b_{m-1}b_m't^m+b_{m-1}c_m't^m
\equiv b_{m-1}c_m + b_m't^m \equiv b_mc_m \; (\mod t^{m+1}\hat R_0)$, 
using that $b_{m-1} \equiv 1 \; (\mod t\hat R_1)$, $c_m \equiv 1 \; (\mod t\hat R_2)$. 
\end{proof}

\begin{prop}[Local Weierstrass Preparation]\label{localweierstrass} 
In the context of Notation~\ref{localnotation}, let $F$ be the fraction field of $\hat R$.  
Then every 
element of $\hat R_1$ is the product of an element of $F$ and a unit in $\hat R_1$.
\end{prop}

\begin{proof}
We may assume $a \in \hat R_1$ is non-zero, and hence $a = t^s a'$ for some non-negative 
integer $s$ and some $a' \in \hat R_1$ that is not divisible by $t$.  Replacing $a$ by $a'$, 
we may assume that $a \not\in t\hat R_1$, and hence that $a$ is a unit in the discrete valuation 
ring $\hat R_0$.  So by Lemma~\ref{localproduct}, 
$a = bc$ for some units $b \in \hat R_1^\times$ and $c \in \hat R_2^\times$, and then $c = ab^{-1} \in \hat R_1$.  
But $\hat R_1 \cap \hat R_2 = \hat R$ by 
Lemma~\ref{generalintersect}, using in particular that 
$\bar R_1 \cap \bar R_2 = \bar R$.  Hence $c \in \hat R_1 \cap \hat R_2 = \hat R \subset F$.
\end{proof}

\begin{thm}\label{localintersection}
In the context of Notation~\ref{localnotation}, let 
$F, F_1, F_2, F_0$ be the fraction fields of $\hat R, \hat R_1, \hat R_2, \hat R_0$ respectively.
Then $F_1 \cap F_2 = F$ in $F_0$.
\end{thm}

\begin{proof}
Let $h \in F_1 \cap F_2$. Write $h = a/b$ with $a,b \in \hat R_1$. By Proposition~\ref{localweierstrass}, $b=uf$ 
for some unit $u \in \hat R_1$ and some non-zero $f \in F$.  Thus $h=au^{-1}/f$; and replacing $h$ 
by $fh$, we may assume $h = au^{-1} \in \hat R_1$.  But $\hat R_2$ is a complete discrete 
valuation ring with uniformizer $t$ (Lemma~\ref{notationlemma}(c)); so after multiplying 
$h \in F_2$ by a non-negative power of $t$ we may assume $h \in \hat R_2 \cap \hat R_1$.  As noted in the above proof, Lemma~\ref{generalintersect} implies that $\hat R_1 \cap \hat R_2 = \hat R \subset F$ and hence $h \in F$.
\end{proof}

%-------------------------------------------------------------------------------
\subsection{Patching}
%-------------------------------------------------------------------------------

We begin with a local patching result, using the above factorization and intersection results.

\begin{thm}\label{localpatching}
In the context of Notation~\ref{localnotation}, let $F$ be the fraction field of $\hat R$ and let $F_i$ be the fraction field of $\hat R_i$ for $i=0,1,2$.   
Then the base change functor 
$$\Vect (F) \to \Vect (F_1) \times_{\Vect (F_0)} \Vect (F_2)$$
is an equivalence of categories.
\end{thm}

\begin{proof}
This follows from Theorem~\ref{localfactorization} and Theorem~\ref{localintersection},
by Proposition~\ref{harb}.
\end{proof}

Combining this with the global patching result Theorem~\ref{globalpatching}, 
we obtain the following result on complete local/global patching:

\begin{thm}\label{clpatching}
Let $T$ be a complete discrete valuation ring with uniformizer $t$, 
and let $\hat X$ be a smooth connected
projective $T$-curve with closed fibre $X$.
Let $\hat R_Q$ be the completion of the local ring of $\hat X$ at a closed point $Q$; let
$\hat R_Q^\circ$ be the $t$-adic completion 
of the localization of $\hat R_Q$ at the height one prime $t\hat R_Q$; and let $F_Q$, $F_Q^\circ$ be the fraction fields of $\hat R_Q$, $\hat R_Q^\circ$.  Let $U$ be a subset of $X$ that contains $Q$, let
$U'  = U \smallsetminus \{Q\}$, and let $F_U$ and $F_{U'}$ 
be as in Notation~\ref{globalnotation}.
Then the base change functor 
$$\Vect (F_U) \to \Vect (F_Q) \times_{\Vect (F_Q^\circ)} \Vect (F_{U'})$$ is an equivalence of categories.
\end{thm}

\begin{proof}
As in Notation~\ref{globalnotation}, we let $\hat R_\varnothing$ be the $t$-adic completion of the local ring of $\hat X$ at the generic point 
of $X$ and let $F_\varnothing$ be the fraction field of $\hat R_\varnothing$.  
Here $\hat R_Q$ denotes the completion of the local ring of $\hat X$ at $Q$ with respect to its maximal ideal, whereas
$\hat R_{\{Q\}}$ denotes the $t$-adic completion of this same local ring.  Also, $F_Q$, $F_{\{Q\}}$ denote the fraction fields of $\hat R_Q$, $\hat R_{\{Q\}}$.  

Since $\hat X$ is a smooth projective $T$-curve, $\hat R_{\{Q\}}$ is a two-dimensional regular local domain, with maximal ideal $\m_Q$, and with $\m_Q$-adic completion $\hat R_Q$.  Choose a lift $f \in \hat R_{\{Q\}}$ of a uniformizer $\bar f$ of $Q$ on the closed fibre $X$; thus $f,t$ form a system of local parameters for $\hat X$ at $Q$.  The localization $\hat R_{\{Q\}}[f^{-1}]_{(t)}$ contains the local ring $R_\varnothing$ of $\hat X$ at the generic point of $X$ and is contained in the $t$-adic completion $\hat R_\varnothing$ of $R_\varnothing$; hence $\hat R_\varnothing$ is the $t$-adic completion of $\hat R_{\{Q\}}[f^{-1}]_{(t)}$, or equivalently of $\hat R_{\{Q\}}[f^{-1}]$.  Similarly, the localization $\hat R_Q[f^{-1}]_{(t)}$ contains the localization $(\hat R_Q)_{(t)}$, and is contained in the $t$-adic completion $\hat R_Q^\circ$ of that ring; hence $\hat R_Q^\circ$ is in fact the $t$-adic completion of $\hat R_Q[f^{-1}]_{(t)}$, or equivalently of $\hat R_Q[f^{-1}]$.  Thus the four rings $\hat R_{\{Q\}}$, $\hat R_Q$, $\hat R_\varnothing$, $\hat R_Q^\circ$ satisfy the assumptions of Notation~\ref{localnotation} for the rings
$\hat R$, $\hat R_1$, $\hat R_2$, $\hat R_0$ there.  So by
Theorem~\ref{localpatching}, the base change functor
$$\Vect (F_{\{Q\}}) \to \Vect (F_Q) \times_{\Vect (F_Q^\circ)} \Vect (F_\varnothing)
$$
is an equivalence of categories.  
By Theorem~\ref{globalpatching}, the base change functor
$$\Vect (F_U) \to \Vect (F_{\{Q\}}) \times_{\Vect (F_\varnothing)} \Vect (F_{U'})
$$
is also an equivalence.  Hence the composition
\begin{align*}
\Vect (F_U) &\to \Vect (F_{\{Q\}}) \times_{\Vect (F_\varnothing)} \Vect (F_{U'})\cr 
&\to \left(\Vect (F_Q) \times_{\Vect (F_Q^\circ)} \Vect (F_\varnothing)\right) \times_{\Vect
  (F_\varnothing)} \Vect (F_{U'})\cr
&\to \Vect (F_Q) \times_{\Vect (F_Q^\circ)} \Vect (F_{U'}),
\end{align*}
given by base change, is an equivalence of categories.
\end{proof}

To illustrate the above result, let $T=k[[t]]$, let $\hat X$ be the projective $x$-line over $T$, let $Q$ be the point $x=t=0$, let $U = \P^1_k$, and let $U' = U \smallsetminus \{Q\}$.  The ring $R_{U'}$ 
contains $k[[t]][x^{-1}]$, the ring of regular functions on an affine open subset $\tilde U'$ of $\P^1_T$; and it is equal to the intersection of the localizations of $k[[t]][x^{-1}]$ at the maximal ideals corresponding to the points of $U'$ (i.e.\ the closed points of the closed fibre of $\tilde U'$).  So $R_{U'}$
is the localization $S^{-1}(k[[t]][x^{-1}])$, where $S$ is the multiplicative set of elements that lie in none of these maximal ideals, or equivalently
are units modulo $t$ (and hence modulo $t^n$ for all $n$).  Thus 
the inclusion $k[[t]][x^{-1}] \hookrightarrow 
R_{U'}$ becomes an isomorphism modulo $t^n$ for all $n$.  Hence $k[[t]][x^{-1}]$ is $t$-adically dense in $R_{U'}$, and these two rings have the same $t$-adic completion; i.e.\ $k[x^{-1}][[t]] = \hat R_{U'}$.  The local ring of $\hat X$ at $Q$ is $k[[t]][x]_{(x,t)}$, whose $t$-adic completion $\hat R_{\{Q\}}$ is $k[x]_{(x)}[[t]]$, this being the inverse limit of the mod $t^n$-reductions $k[t,x]_{(x)}/(t^n)$.  The $(x,t)$-adic completion of this same local ring is 
$\hat R_Q = k[[x,t]]$; while 
$\hat R_\varnothing = k(x)[[t]]$ and $\hat R_Q^\circ = k((x))[[t]]$, these being the $t$-adic completions of $k[[t]][x]_{(x)}$ and $k[[x,t]][x^{-1}]$.  The fields $F_{U'}$, $F_{\{Q\}}$, $F_Q$, $F_\varnothing$, and $F_Q^\circ$ are the respective fraction fields of the above complete rings.  The function field $F_U$ of $\hat X$ is the fraction field $k((t))(x)$ of $k[[t]][x]$, the ring of functions on the dense open subset $\A^1_T$.

The next result is a generalization of Theorem~\ref{clpatching} that allows more patches.  

\begin{thm}\label{localpatchingseveral}
Let $T$ be a complete discrete valuation ring with uniformizer $t$, 
and let $\hat X$ be a smooth connected 
projective $T$-curve with closed fibre $X$.
Let $Q_1,\dots,Q_r$ be distinct
closed points on $\hat X$.  For each $i=1,\dots,r$, let $\hat R_i$ be the complete local ring of $\hat X$ at $Q_i$; let $\hat R_i^\circ$ be the $t$-adic completion 
of the localization of $\hat R_i$ at the height one prime $t\hat R_i$; and let $F_i$, $F_i^\circ$ be the fraction fields of $\hat R_i$, $\hat R_i^\circ$.  Let $U$ be a subset of $X$ that contains $S = \{Q_1,\dots,Q_r\}$, let
$U' = U \smallsetminus S$, and let $F_U$ and $F_{U'}$ 
be as in Notation~\ref{globalnotation}.
Then the base change functor 
$$\Vect (F_U) \to \prod_{i=1}^r \Vect (F_i) \times_{\prod_{i=1}^r \Vect (F_i^\circ)} \Vect (F_{U'})$$ is an equivalence of categories.
\end{thm}

\begin{proof}
This follows by induction from Theorem~\ref{clpatching}, using the identification of
$$\prod_{i=1}^{j-1} \Vect(F_i) \times_{\prod_{i=1}^{j-1} \Vect (F_i^\circ)} \bigl(
  \Vect(F_j) \times_{\Vect(F_j^\circ)} \Vect(F_{U \smallsetminus  \{Q_1,\dots,Q_j\}})\bigr)$$
with
$$\prod_{i=1}^j \Vect(F_i) \times_{\prod_{i=1}^j \Vect (F_i^\circ)} \Vect(F_{U \smallsetminus  \{Q_1,\dots,Q_j\}}).$$
\end{proof}

In the terminology of Section~\ref{setup}, we may rephrase the above result in terms of patching problems.  Namely, consider the partially ordered set $I=\{1,\dots,r,1',\dots,r',U'\}$, where $i \succ i'$ for each $i$, and where $U' \succ i'$ for all $i$.  Set $F_{i'} = F_i^\circ$ for each $i$, and consider the corresponding finite inverse system of fields $\F=\{F_i,F_{i'},F_{U'}\}$ indexed by $I$.
Then Theorem~\ref{localpatchingseveral} asserts that the base change functor $\Vect(F_U) \to \PP(\F)$ is an equivalence of categories.

\begin{remark} \label{analogy}
The above theorem can be regarded as analogous to a special case of 
Theorem~\ref{globalpatchingseveral} --- viz.\ where each of the sets $U_i$ consists of a single point, except for one $U_i$ which is disjoint from the others.  Both results then make a patching assertion in the context of one arbitrary set and a finite collection of points not in that set.  The main difference between the two results is that in the above special case of Theorem~\ref{globalpatchingseveral}, the local patches correspond to the fraction fields of the $t$-adic completions of the local rings at the respective points $Q_i$; whereas Theorem~\ref{localpatchingseveral} uses the fraction fields of the $\m_{Q_i}$-adic completions of the local rings at those points.  
In the special case of Theorem~\ref{globalpatchingseveral}, the 
``overlap'' fields associated to the pairwise intersections $U_i \cap U_j$ are each just the fraction field $F_\varnothing$ of the $t$-adic completion of the local ring at the generic point of $X$; whereas in 
Theorem~\ref{localpatchingseveral}, the ``overlap'' fields $F_i^\circ$ are different from each other (and are larger than $F_\varnothing$).
So Theorem~\ref{localpatchingseveral} 
can no longer be phrased as a fibre product over a single base as in Theorem~\ref{globalpatchingseveral}.  
\end{remark}

\begin{prop} \label{localinverse}
In Theorems~\ref{localpatching},  \ref{clpatching}, and \ref{localpatchingseveral},
the inverse of the base change functor (up to isomorphism) is given by taking the inverse limit of the vector spaces on the patches.  In Theorems~\ref{localpatching} and \ref{clpatching}, this inverse limit is given by taking the intersection of vector spaces.
\end{prop}

\begin{proof}
By Corollary~\ref{dimcriterion}, the assertion for Theorems~\ref{localpatching} and \ref{clpatching}
follows from verifying the intersection condition of Section~\ref{setup} concerning fields (i.e.\ that $F_1 \cap F_2 = F$ in Theorem~\ref{localpatching} and that $F_Q \cap F_{U'} = F_U$ in Theorem~\ref{clpatching}).  That condition follows for Theorem~\ref{localpatching} by Theorem~\ref{localintersection}; and for Theorem~\ref{clpatching} by combining that in turn with Theorem~\ref{globalintersection}.

To prove the result in the case of Theorem~\ref{localpatchingseveral}, we rephrase that theorem as asserting that $\Vect(F_U) \to \PP(\F)$ is an equivalence of categories, where $\F$ is as defined in the paragraph following the proof above.  By Theorem~\ref{localintersection}, $F_Q \cap F_{U'} = F_U$ in the situation of Theorem~\ref{clpatching}.  So by induction on $r$, it follows that $F_U$ is the inverse limit of the fields in $\F$.  Hence Proposition~\ref{patchprob} asserts that if a patching problem $\mathcal{V}=\{V_i,V_{i'},V_{U'}\}$ in $\F$ is induced (up to isomorphism) by a finite dimensional $F_U$-vector space $V$, then $V$ is isomorphic to the inverse limit of  $\mathcal{V}$.  Such a $V$ exists since the functor in Theorem~\ref{localpatchingseveral} is an equivalence of categories; hence the assertion follows. 
\end{proof}

%%%%%%%%%%%%%%%%%%%%%%%%%%%%%%%%%%%%%%%%%%%%%%%%%%%%%%%%%%%%%%%%%%%%%%%%%%%%%%%%
\section{Allowing Singularities} \label{singular}
%Section 6
%%%%%%%%%%%%%%%%%%%%%%%%%%%%%%%%%%%%%%%%%%%%%%%%%%%%%%%%%%%%%%%%%%%%%%%%%%%%%%%%
In view of later applications, it is desirable to have a version of
Theorem~\ref{localpatchingseveral} that can be applied to a singular curve.
Let $T$ be a complete discrete valuation ring with uniformizer $t$.
In order to perform patching in the case of normal curves 
$\hat X \to T$ that are not smooth, we introduce some terminology that was used
in a related context in \cite{harbaterstevenson}, Section~1.  

Let $\hat X$ be a connected projective normal $T$-curve, with closed fibre $X$.  Consider a non-empty finite set
$S \subset X$ that contains every point where distinct irreducible components of $X$ meet.  Thus each connected component of $X \smallsetminus S$ is contained in an irreducible component of $X$, and moreover is an affine open subset of that component, since each irreducible component of $X$ contains at least one point of $S$ (by connectivity of $X$ and the fact that $X \ne \varnothing$).  For any non-empty
irreducible affine Zariski open subset $U \subseteq X\smallsetminus S$, we consider as before the ring
$R_U$ of  rational functions on $\hat X$ that are regular at the points of $U$ (and hence also at the generic point of the component of $X$ containing $U$); and
the fraction field  $F_U$ of the $t$-adic completion $\hat R_U$ of $R_U$ (which is a domain by the irreducibility of $U$).  For
each point $P \in S$,
the complete local ring $\hat R_P$ of $\hat X$ at $P$ is a domain, say
with fraction field $F_P$.  Each height one prime ideal $\wp$ of
$\hat R_P$  that contains $t$ determines a {\it branch} of $X$ at $P$ (i.e.\ an irreducible component of the pullback of $X$ to $\Spec \hat R_P$); 
and we
let $\hat R_{\wp}$  denote the complete local ring of $\hat R_P$ at $\wp$,
with fraction field  $F_{\wp}$.  Since $t \in \wp$, the contraction of $\wp \subset \hat R_P$ to the local ring $\O_{\hat X, P}$ defines an irreducible component of $\Spec \O_{X, P}$; hence an irreducible component of $X$ containing $P$.
This in turn is the closure of a unique connected component $U$ of $X\smallsetminus S$; and we
say that $\wp$  {\bf lies on} $U$.  (Note that several branches of $X$ at $P$
may  lie on the same $U$, viz.\ if the closure of $U$ is not unibranched at $P$.)  

In this situation, we obtain a finite inverse system consisting of fields $F_U, F_P, F_\wp$.  More precisely, let $I_1$ be the set of irreducible (or equivalently, connected) components $U$ of $X\smallsetminus S$; let $I_2=S$; let $I_0$ be the set of branches $\wp$ of $X$ at points $P \in S$; and let $I = I_1 \cup I_2 \cup I_0$.  Give $I$ the structure of a partially ordered set by setting $U \succ \wp$ if $\wp$ lies on $U$, and setting $P \succ \wp$ if $\wp$ is a branch of $X$ at $P$.  This defines the asserted finite inverse system $\F = \F_{\hat X, S} = \{F_i\}_{i \in I}$ consisting of the fields $F_U, F_P, F_\wp$ under the natural inclusions $F_U \to F_\wp$ and $F_P \to F_\wp$.

Theorem~\ref{singularpatching} below, which states a patching result for singular curves, will be proven by relating a given singular curve to an auxiliary smooth curve.  That theorem and the results preceding it will be in the following situation:

\begin{hypothesis} \label{singpatchhyp}  We make the following assumptions: 
\begin{itemize}
\item $T$ is a complete discrete valuation ring with uniformizer $t$.
\item $f:\hat X \to \hat X'$ is a finite morphism of connected projective normal $T$-curves with function fields $F,F'$, and closed fibres $X, X'$, such that $\hat X'$ is smooth over $T$.
\item $S' \ne \varnothing$ is a finite set of closed points of $X'$, say with complement $U' = X' - S'$, such that 
$S := f^{-1}(S') \subset X$ contains the points where distinct irreducible components of $X$ meet.
\end{itemize}
\end{hypothesis}

\begin{lemma} \label{singisos}
Under Hypothesis~\ref{singpatchhyp} and the above notation, let 
$P' \in S'$ and let $\wp'$ be the branch of $X'$ at $P'$.  
\renewcommand{\theenumi}{\alph{enumi}}
\begin{enumerate}
\item The natural maps 
$$F \otimes_{F'} F_{U'} \to \prod F_U, \ \ \
F \otimes_{F'} F_{P'} \to \prod F_P, \ \ \
F \otimes_{F'} F_{\wp'} \to \prod F_\wp$$
are isomorphisms of $F$-algebras, where the products respectively range over the connected components $U$ of $X \smallsetminus S$, the points $P$ of $S$ lying over $P'$, and the branches $\wp$ of $X$ over $\wp'$.  \label{singisos1}
\item The natural inclusions $F_{U'} \to F_{\wp'}$ and $F_{P'} \to F_{\wp'}$ are compatible with the natural inclusions
$\prod F_U \to \prod F_\wp$ and $\prod F_P \to \prod F_\wp$, where $U$ and $P$ range as above and $\wp$ ranges over the branches of $X$ at points of $S$.  \label{singisos2}
\end{enumerate}
\end{lemma}

\begin{proof} (\ref{singisos1})
Choose a Zariski affine open subset $\Spec \tilde R'$ of $\hat X'$ that meets $X'$ in $U'$, and let $\Spec \tilde R$ be its inverse image in $\hat X$.  The fraction fields of $\tilde R, \tilde R'$ are respectively the function fields $F, F'$ of $\hat X, \hat X'$, since $\Spec \tilde R\subset \hat X$, $\Spec \tilde R\subset \hat X'$ are Zariski open dense subsets. 

At every closed point of $U'$, the rings $\tilde R'$ and $R'_{U'}$ have the same localization and hence the same completion; so $R'_{U'}$ and its subring $\tilde R'$ have the same $t$-adic completion $\hat R_{U'}$.  By the hypothesis on $S$, the inverse image 
$f^{-1}(U') = X \smallsetminus S \subset X$
is the disjoint union of its irreducible components $U$.  Thus the  
corresponding ideals $I_U$ of $\tilde R$ are pairwise relatively prime; hence $t\tilde R = \prod I_U$ and more generally $t^n\tilde R = \prod I_U^n$ for $n \ge 1$.  By the Chinese Remainder Theorem, $\tilde R/t^n\tilde R$ is isomorphic to the product $\prod_U \tilde R/I_U^n = \prod_U R_U/I_U^nR_U= \prod_U R_U/t^nR_U$; and taking inverse limits shows that the $t$-adic completion of $\tilde R$ is  
$\prod \hat R_U$.  By definition, the
fraction fields of $\hat R_U, \hat R_{U'}$ are respectively $F_U, F_{U'}$, where as above $U$ ranges 
over the connected components of $X \smallsetminus S$.

The natural $\tilde R$-algebra homomorphism $\tilde R \otimes_{\tilde R'} \hat R_{U'} \to \prod \hat R_U$ is bijective, by \cite{bourbaki}, Theorem~3(ii) in \S3.4 of Chapter~III, and thus is an isomorphism.  Hence so is $\tilde R \otimes_{\tilde R'} F_{U'} = 
\tilde R \otimes_{\tilde R'} \hat R_{U'} \otimes_{\hat R_{U'}} F_{U'}
\to \prod \hat R_U \otimes_{\hat R_{U'}} F_{U'}$. 
Now $\tilde R$ is finite over $\tilde R'$, and 
$\hat R_U$ is finite over $\hat R_{U'}$;
so $F = \tilde R \otimes_{\tilde R'} F'$
and $F_U = \hat R_U \otimes_{\hat R_{U'}} F_{U'}$.  Hence $\prod F_U = \prod \hat R_U \otimes_{\hat R_{U'}} F_{U'}$ as $F$-algebras.  
Thus the natural map $F \otimes_{F'} F_{U'} 
= \tilde R \otimes_{\tilde R'} F' \otimes_{F'}  F_{U'}
= \tilde R \otimes_{\tilde R'} F_{U'} 
\to \prod F_U$ is an $F$-algebra isomorphism.  
This proves that the first map is an isomorphism.  The proofs for the other two maps are similar.

(\ref{singisos2}) This follows from the fact that each of these maps is given by base change.
\end{proof}

In the situation of Hypothesis~\ref{singpatchhyp}, 
consider the diagonal inclusion map
$F \to \prod F_U \times \prod F_P$, where $U$ ranges over the components of $X \smallsetminus S$ and $P$ ranges over $S$.  For each such $U$ consider the diagonal inclusion
$\iota_U:F_U \to \prod F_\wp$, where 
$\wp$ ranges over branches of $X$ lying on $U$; and for each $P$ consider the diagonal inclusion
$\iota_P:F_P \to \prod F_\wp$, where
$\wp$ ranges over branches of $X$ at $P$.  
Consider the sum $\prod F_U \times \prod F_P \to  \prod F_{\wp}$ of the maps $\iota_U$ and $-\iota_P$ on the respective components; here $\wp$ ranges over all the branches at points of $S$.

\begin{prop} \label{exseqsing}
Under Hypothesis~\ref{singpatchhyp} and the above notation,
the sequence
$$0 \to F \to \prod F_U \times \prod F_P  \to  \prod F_{\wp}$$
of $F$-vector spaces is exact.  Equivalently, $F$ is the inverse limit of the inverse system $\F_{\hat X,S}$ of $F$-algebras consisting of the fields $F_U, F_P, F_\wp$ with the natural inclusions.
\end{prop}

\begin{proof}
With respect to the inclusions $F_{U'} \to F_{\wp'}$ and $\iota_{P'}:F_{P'} \to  F_{\wp'}$, we obtain an inverse system $\F' = \F_{\hat X', S'}$ of fields $F_{U'}, F_{P'}, F_{\wp'}$, where $P'$ ranges over $S'$ and $\wp'$ ranges over the corresponding branches.  
Applying Proposition~\ref{localinverse} in the situation of Theorem~\ref{localpatchingseveral},
the function field $F'$ of $\hat X'$ is the inverse limit of the system $\F'$, viewing each of the fields fields $F_{U'}, F_{P'}, F_{\wp'}$ as a one-dimensional vector space over itself.
Writing $F' \to F_{U'} \times \prod F_{P'}$ and $F_{U'} \to \prod F_{\wp'}$ for the diagonal inclusions, and writing $\prod F_{P'} \to  \prod F_{\wp'}$ for the product of the maps $-\iota_{P'}$, 
the inverse limit assertion for $\F'$ is equivalent to the exactness of the sequence of $F'$-vector spaces
$$0 \to F' \to F_{U'} \times \prod F_{P'} \to \prod F_{\wp'}.$$

The desired exactness now follows from tensoring this exact sequence over $F'$ with $F$, and using Lemma~\ref{singisos}.  This exactness is then equivalent to the assertion that $F$ is the inverse limit of the system $\F_{\hat X,S}$.  \end{proof}

Note that in the above result, as in Proposition~\ref{localinverse} in the situation of Theorem~\ref{localpatchingseveral}, we must phrase the assertion in terms of inverse limits rather than intersections, because the various fields in the inverse system are not all contained in some common overfield in the system.  

Under Hypothesis~\ref{singpatchhyp} and the above notation, we
define a {\bf (field) patching problem} $\mathcal{V}$ for $(\hat X, S)$ to be a patching problem (in the sense of Section~\ref{setup}) for the inverse system $\F=\F_{\hat X,S}$.  Because of the form of the index set of $\F$, and as noted in the last paragraph of Section~\ref{setup}, giving such a patching problem is equivalent to giving:

\renewcommand{\theenumi}{\roman{enumi}}
\begin{enumerate}
\item a finite dimensional $F_U$-vector space $V_U$ for every irreducible component $U$ of $X\smallsetminus S$;\label{pp1}
\item a finite dimensional $F_P$-vector space $V_P$ for every $P \in S$;\label{pp2}
\item an $F_{\wp}$-vector space isomorphism $\mu_{U,P,\wp}:V_U \otimes_{F_U}
  F_{\wp}  \iso V_P  \otimes_{F_P} F_{\wp}$ for each choice of $U,P,\wp$,
  where $U$ is  an irreducible component of $X\smallsetminus S$; $P \in S$ is in the closure
  of $U$; and $\wp$  is a branch of $X$ at $P$ that lies on $U$.\label{pp3}
\end{enumerate}

We write $\PP(\hat X,S)$ for the category $\PP(\F)$ of patching problems for $(\hat X, S)$ (or equivalently, for $\F$).  The function field $F$ of $\hat X$ is contained in each $F_i$ for $i \in I$, and in fact $F$ is the inverse limit of the $F_i$ by Proposition~\ref{exseqsing}.

By the above containments, every finite dimensional $F$-vector space $V$ induces a patching problem $\beta_{\hat X, S}(V)$ for $(\hat X, S)$ via base change, and $\beta_{\hat X, S}$ defines a functor from $\Vect(F)$ to $\PP(\hat X,S)$.  There is also a functor $\iota_{\hat X, S}$ from 
$\PP(\hat X,S)$ to $\Vect(F)$ that assigns to each patching problem its inverse limit (which we view as the ``intersection'', though as in the case of the fields there is in fact no common larger object within which to take an intersection).

The following result is similar to Theorem~1(a) of \cite{harbaterstevenson},
\S 1, which considered a related notion of
patching problems for rings and modules rather than for fields and vector spaces.

\begin{thm} \label{singularpatching}
Under Hypothesis~\ref{singpatchhyp}, the base change functor $\beta_{\hat X,S}: \Vect(F) \to \PP(\hat X,S)$ is an equivalence of categories, and $\iota_{\hat X, S}\beta_{\hat X, S}$ is isomorphic to the identity functor on $\Vect(F)$.
\end{thm}

\begin{proof}
The equivalence of categories assertion is that $\beta_{\hat X,S}$ is surjective on isomorphism classes; and that the natural maps $\Hom_F(V_1,V_2) \to \Hom_F(\beta_{\hat X,S}(V_1),\beta_{\hat X,S}(V_2))$ are bijective for $V_1,V_2$ in $\Vect(F)$.  We show that these hold, and that $\iota_{\hat X, S}\beta_{\hat X, S}$ is isomorphic to the identity functor, in steps (with the surjectivity requiring the bulk of the work).

\medskip

{\it Step~1:} To show that $\beta_{\hat X,S}$ is surjective on isomorphism classes.

\smallskip

As in the discussion before the theorem, a patching problem $\mathcal{V}$ for $(\hat X,S)$ corresponds to a collection of finite-dimensional $F_U$-vector spaces $V_U$ and $F_P$-vector spaces $V_P$ together with isomorphisms $\mu_{U,P,\wp}$. Let $W_{U'} = \prod V_U$ (ranging over components $U$ of $X \smallsetminus S$); and for $P' \in S'$ let  $W_{P'} = \prod V_P$, where $P$ ranges
over $S_{P'} := f^{-1}(P') \subseteq S$.  Since the fields $F_U$ and $F_P$ are respectively finite over $F_{U'}$ and $F_{P'}$ (where $f(P)=P'$), 
$W_{U'}$ is a finite dimensional vector space over $F_{U'}$, and $W_{P'}$ is a  
finite dimensional vector space over $F_{P'}$.  

For each $P' \in S'$ with associated branch $\wp'$, we have identifications 
$\prod F_U \otimes_{F_{U'}} F_{\wp'} 
= (F \otimes_{F'} F_{U'}) \otimes_{F_{U'}} F_{\wp'}
= F \otimes_{F'} F_{\wp'}
= \prod F_\wp$ of $F_{\wp'}$-algebras, where the last product ranges over the branches $\wp$ lying over $\wp'$.  Using 
the algebra isomorphisms 
$F \otimes_{F'} F_{U'} \ \iso \  \prod F_U$ and $F \otimes_{F'} F_{\wp'} \ \iso \  \prod F_\wp$
from Lemma~\ref{singisos},
we thus obtain identifications
$W_{U'} \otimes_{F_{U'}} F_{\wp'} = (\prod V_U) \otimes_{F_{U'}} F_{\wp'}
= (\prod V_U) \otimes_{\prod F_U} (\prod F_U \otimes_{F_{U'}} F_{\wp'})
= (\prod V_U) \otimes_{\prod F_U} \prod F_\wp
= \prod_U (V_U \otimes_{F_U} \prod F_\wp)
= \prod_U \prod_\wp V_U \otimes_{F_U} F_\wp$
of $F_{\wp'}$-vector spaces, with the products ranging over components $U$ over $U'$, and branches $\wp$ lying on $U$ (and lying over $\wp'$).  Similarly, using the algebra isomorphisms 
$F \otimes_{F'} F_{P'} \ \iso \  \prod F_P$ and
$F \otimes_{F'} F_{\wp'} \ \iso \  \prod F_\wp$
from Lemma~\ref{singisos},
we obtain an identification $W_{P'} \otimes_{F_{P'}} F_{\wp'} 
= \prod_P \prod_{\wp} V_P \otimes_{F_P} F_\wp$ of $F_{\wp'}$-vector spaces, where the products range over points $P \in S_{P'}$ and branches $\wp$ lying over $\wp'$.   
Combining the above identifications with the product of the isomorphisms 
$\mu_{U,P,\wp}:V_U \otimes_{F_U} F_{\wp} \ \iso\ 
V_P  \otimes_{F_P} F_{\wp}$ for $P \in S$ over $P' \in S'$, we obtain an isomorphism
$\mu'_{U',P',\wp'}:W_{U'} \otimes_{F_{U'}} F_{\wp'} 
\ \iso \ W_{P'} \otimes_{F_{P'}} F_{\wp'}$
of $F_{\wp'}$-vector spaces.
  
As in the discussion before the theorem, the vector spaces $W_{U'}, W_{P'}$ together with the $F_{\wp'}$-isomorphisms $\mu'_{U',P',\wp'}$ define a patching problem $\mathcal{W} =: f_*(\mathcal{V})$ for $(\hat X', S')$.
By Theorem~\ref{localpatchingseveral}, there is a finite dimensional $F'$-vector space $W$ which is a solution to the patching problem $\mathcal{W}$; i.e., $\mathcal{W} = \beta_{\hat X',S'}(W)$.  
In order to complete the proof of the surjectivity of $\beta_{\hat X,S}$ on isomorphism classes of objects, it will suffice to give $W$ the structure of an $F$-vector space and to show that with respect to this additional structure, $\beta_{\hat X,S}(W)$ is isomorphic to the given patching problem
$\mathcal{V}$ for $(\hat X,S)$.  

To do this, consider the ``identity patching problem'' 
$\beta_{\hat X, S}(F)$ for
$(\hat X,S)$,  given by $F_U$, the $F_P$, and the identity maps on each
$F_{\wp}$.  Let $f_*(F)$ denote $F$ viewed as an $F'$-vector space; similarly let $f_*(F_U)$, $f_*(F_P)$ denote $F_U$, $F_P$ as vector spaces over $F_{U'}$, $F_{P'}$ respectively.  
The patching problem $\beta_{\hat X',S'}(f_*(F))$ for $(\hat X',S')$ induced by $f_*(F)$ is thus given by $f_*(F_U)$, the $f_*(F_P)$, and the identity map on $F_{\wp'}$.  Let 
$\alpha_{U}:f_*(F_{U}) \to \End_{F_{U'}}(W_{U})$  and $\alpha_{P}:f_*(F_{P}) \to \End_{F_{P'}}(W_{P})$ (for $P \in S_{P'}$) be the maps corresponding to scalar multiplication by $F_U$ and $F_P$ on the factors $V_U$ of $W_{U'}$ and the factors $V_P$ of $W_{P'}$, respectively.  These maps define a morphism in the
category of patching  problems $\bar \alpha:\beta_{\hat X',S'}(f_*(F)) \to \beta_{\hat X',S'}(\End_{F'}(W))$ for $(\hat X',S')$.
By the equivalence
of categories assertion in Theorem~\ref{localpatchingseveral} 
applied to the $T$-curve $\hat X'$ and finite subset $S'$,
the element  $\bar \alpha \in
\Hom_{F'}(\beta_{\hat X',S'}(f_*(F)),\beta_{\hat X',S'}(\End_{F'}(W))$ is induced by a unique morphism $\alpha
\in \Hom_{F'}(f_*(F),\End_{F'}(W))$  in the category of finite dimensional $F'$-vector
spaces.   As a result, $W$ is given the structure of a finite dimensional
$F$-vector space,  with $\alpha$ defining scalar multiplication.  
It is now straightforward to check that $\beta_{\hat X,S}(W)$ is isomorphic to
$\mathcal{V}$,  showing the desired surjectivity on isomorphism classes.

\medskip

{\it Step~2:}  To show that $\iota_{\hat X,S}\beta_{\hat X,S}$ is isomorphic to the identity functor.

\smallskip

For any $V$ in $\Vect(F)$, the induced patching problem
$\beta_{\hat X,S}(V)$ corresponds to data $V_U, V_P, \mu_{U,P,\wp}$.
Tensoring the exact sequence in Proposition~\ref{exseqsing} over $F$ with $V$ gives an exact sequence $$0 \to V \to \prod V_U \times \prod V_P \to \prod V_{\wp}$$ 
of $F$-vector spaces.  Here $V_{\wp} := V_P \otimes_{F_P} F_{\wp}$ for $\wp$
a  branch of $X$ at $P$; $V_U \to V_{\wp}$ is defined via $\mu_{U,P,\wp}$;
and $V_P \to V_\wp$ is minus the natural inclusion.
This shows that $V$ is naturally isomorphic to $\iota_{\hat X,S}(\beta_{\hat X,S}(V))$; i.e.,
$\iota_{\hat X,S}\beta_{\hat X,S}$ is isomorphic to the identity functor on $\Vect(F)$.  

\medskip

{\it Step~3:}  To show that $\beta_{\hat X,S}$ induces a bijection between maps between corresponding objects.

\smallskip

Consider $V_1,V_2$ in $\Vect(F)$, with induced patching problems $\beta_{\hat X,S}(V_1),\beta_{\hat X,S}(V_2)$.  Then $V_i \to V_{i,U} := V_i \otimes_F F_U$ 
and $V_i \to V_{i,P} := V_i \otimes_F F_P$ are
inclusions for $i=1,2$, for all $U$ and $P$; and a set of
compatible maps $V_{1,U} \to V_{2,U}$ and $V_{1,P} \to V_{2,P}$ determines a unique
map  $\iota_{\hat X,S}(\beta_{\hat X,S}(V_1)) \to \iota_{\hat X,S}(\beta_{\hat X,S}(V_2))$.
So the natural map from $\Hom_F(V_1,V_2)$ to $\Hom_F(\beta_{\hat X,S}(V_1),\beta_{\hat X,S}(V_2))$ is bijective (and this concludes the proof that $\beta_{\hat X,S}$ is an equivalence of categories). 
\end{proof}

Thus with $\hat X$ and $S$ as in the theorem, every patching problem for
$(\hat X, S)$
has a unique solution up to isomorphism, and this solution is given by the inverse limit of the fields defining the patching problem.

\begin{remark}
Given a non-empty finite subset $S' \subset X'$, the hypothesis on $S = f^{-1}(S')$ is satisfied if $S$ contains every point at which $X$ is not unibranched.  In particular, it is satisfied if $S$ contains all the points at which the reduced structure of $X$ is not regular.  
\end{remark}

In order to apply Theorem~\ref{singularpatching} to $T$-curves $\hat X$ that are not necessarily given in the context of Hypothesis~\ref{singpatchhyp}, we prove the next result.

\begin{prop}  \label{maptoline}
Let $\hat X$ be a projective curve over a discrete valuation ring $T$, having closed fibre $X$.  Let
$S$ be a finite set of closed points of $\hat X$.  Then there is a finite $T$-morphism $f:\hat X \to \P^1_T$ such that $S \subseteq f^{-1}(\infty)$.
\end{prop}

\begin{proof}
Let $t$ be a uniformizer of $T$.
Since $\hat X$ is projective, we may fix an embedding $\hat X \hookrightarrow \P^{\;\!n}_T$ for some $n \ge 1$.  We proceed by induction on $n$.  If $n=1$ then $\hat X = \P^1_T$.  Choosing a rational function $f$ on $\P^1_T$ with poles at each point of $S \subset \hat X = \P^1_T$ yields the result in this case.

So assume $n>1$ and that the result holds for $n-1$.  
Pick a closed point $P \in \P^{\;\!n}_T \smallsetminus \hat X$.  Its residue field $k'$ is a finite extension of the residue field $k$ of $T$; and $P$ is defined over $k'$, say with homogeneous coordinates $(a_0:\cdots:a_n)$ where each $a_i \in k'$.  Possibly after permuting the coordinates, we may assume that $a_0 \ne 0$.  
For each $i \ne 0$, let $g_i(x) \in k[x]$ be the monic minimal polynomial of $a_i/a_0 \in k'$ over $k$.  Since $k$ is the residue field of $T$, we may choose monic polynomials $G_i(x) \in T[x]$ whose reductions modulo $tT(x)$ are $g_i(x)$.  Let $d_i$ be the degree of $G_i$ (or of $g_i$); let $d$ be the least common multiple of $d_1,\dots,d_n$; and let $H_i(x,y)=y^d G_i(x/y)^{d/d_i} \in T[x,y]$.  Thus $H_i$ is a homogeneous polynomial that has total degree $d$ in $x,y$, and is congruent to $x^d$ modulo $yT[x,y]$.  Letting  $m_0,\dots,m_N$ be the distinct (unordered) monomials of degree $d$ in $x_0,\dots,x_n$, we may write $H_i(x_i,x_0) = \sum_{j=0}^N c_{ij}m_j$ for some elements $c_{ij} \in T$.  Note that the locus of $H_1(x_1,x_0) = \cdots = H_n(x_n,x_0) = 0$ is a closed subset of $\P^{\;\!n}_T$ that meets the closed fibre precisely at $P$, and so is disjoint from $\hat X$.

Define the morphism $p:\hat X \to \P^{\;\!n-1}_T$ by $p(x_0:\cdots:x_n) = (H_1(x_1,x_0):\cdots:H_n(x_n,x_0))$. 
Thus $p$ is the composition of the $d$-uple embedding $\iota_d$ of $\hat X \subset \P^{\;\!n}_T$ into $\P^N_T$ with the projection morphism $\pi:\P^N_T - L \to \P^{\;\!n-1}_T$ defined on the complement of the linear subspace $L \subset \P^N_T$ that is given by the linear forms 
$\sum_{j=0}^N c_{ij}m_j$ on $\P^N_T$ for $i=1,\dots,n$.  Now the $d$-uple embedding $\iota_d$ is finite; and so is the restriction of the projection $\pi$ to the closed subscheme $\iota_d(\hat X) \subset \P^N_T - L$, by Proposition~6 of Chapter~II, Section~7 of \cite{redbook}.  So $p$ is finite.  

Let $\hat X_1 \subseteq \P^{\;\!n-1}_T$ be the image of $p$.  By the inductive hypothesis there is a finite $T$-morphism $f_1:\hat X_1 \to \P^1_T$ such that $\pi(S) \subseteq f_1^{-1}(\infty)$.  So $f_1 \circ \pi:\hat X \to \P^1_T$ is a finite $T$-morphism such that $S \subseteq f^{-1}(\infty)$.
\end{proof}

Using the above proposition, we may apply Theorem~\ref{singularpatching} to a given curve $\hat X$ and a given finite set $S$ after possibly enlarging $S$: 

\begin{cor} \label{refinement}
Let $T$ be a complete discrete valuation ring; let $\hat X$ be a connected projective normal $T$-curve with closed fibre $X$; and let $S$ be a finite set of closed points of $X$.  Then there exists a finite subset $S_1 \subset X$ containing $S$ such that Theorem~\ref{singularpatching} holds for $\hat X$, $S_1$.
\end{cor}

\begin{proof}
After enlarging $S$, we may assume that $S$ contains the (finitely many) closed points of $X$ where distinct irreducible components of $X$ meet.  
By Proposition~\ref{maptoline}, there is a finite morphism $f: \hat X \to \P^1_T$ such that $S \subseteq f^{-1}(\infty)$.  So 
Hypothesis~\ref{singpatchhyp} is satisfied by the data $T$, $f: \hat X \to \P^1_T$, $S' = \{\infty\}$, $S_1 = f^{-1}(S')$.  Thus Theorem~\ref{singularpatching} holds for $\hat X$, $S_1$. 
\end{proof}

%%%%%%%%%%%%%%%%%%%%%%%%%%%%%%%%%%%%%%%%%%%%%%%%%%%%%%%%%%%%%%%%%%%%%%%%%%%%%%%%
\section{Applications} \label{applications}
%Section 7
%%%%%%%%%%%%%%%%%%%%%%%%%%%%%%%%%%%%%%%%%%%%%%%%%%%%%%%%%%%%%%%%%%%%%%%%%%%%%%%%
In this section, we give several short applications of the new version of patching. 

%-------------------------------------------------------------------------------
\subsection{Patching Algebras and Brauer Groups} \label{algebras}
%--------------------------------------------------------------------------------
Our patching results for vector spaces carry over to patching for algebras of
various sorts,  because patching was phrased as an equivalence of categories.  

To be more precise, for a field $F$ we will consider finite dimensional
associative $F$-algebras, with or without a multiplicative identity.  We will
also consider additional structure that may be added to the algebra, e.g.\
commutativity, separability, and being Galois with (finite) group $G$.  A finite
commutative $F$-algebra is assumed to have an identity; and it is separable if and only if it is a product of finitely
many separable field extensions of $F$.  By a $G$-{\bf Galois} $F$-algebra
we will mean a commutative $F$-algebra $E$ together with an $F$-algebra
action of $G$ on $E$ such that the ring of $G$-invariants of $E$ is $F$, and such that the inertia group 
$I_\m \leq G$ at each maximal ideal $\m$ of $E$ is trivial.  Such
an extension is necessarily separable and the $G$-action is necessarily faithful. 
If $E$ is a field, being a $G$-Galois
$F$-algebra is equivalent to being a $G$-Galois field extension.  We will also
consider (finite dimensional) central simple algebras over $F$.

\begin{thm}\label{patchingalgebras}
Under the hypotheses of the patching theorems of Sections~\ref{globalsection}, \ref{localsection}, and~\ref{singular} (Theorems~\ref{globalpatching}, \ref{globalpatchingseveral}, \ref{clpatching}, \ref{localpatchingseveral}, \ref{singularpatching}, patching holds with the category of finite dimensional
vector spaces replaced by any of the following (all assumed finite dimensional
over $F$):  
\renewcommand{\theenumi}{\roman{enumi}}
\begin{enumerate}
\item associative $F$-algebras;\label{alg1}
\item associative $F$-algebras with identity;\label{alg2}
\item commutative $F$-algebras;\label{alg3}
\item separable commutative $F$-algebras;\label{alg4}
\item $G$-Galois $F$-algebras;\label{alg5}
\item central simple $F$-algebras.\label{alg6}
\end{enumerate}
\end{thm}
\begin{proof} We follow the strategy of \cite{CAPS}, Prop.~2.8 (cf.\ also
  \cite{MSRI}, 2.2.4).

The equivalence of categories in each of the patching results of
Sections~\ref{globalsection} and \ref{localsection}
is given by a base change functor $\beta$, which preserves tensor products.  
So $\beta$ is an equivalence of tensor categories.  

An associative $F$-algebra is an $F$-vector space $A$ together with a vector
space  homomorphism $p:A \otimes_F A \to A$ that defines the product and
satisfies  an identity corresponding to the associative law.  Since the base
change  patching functor $\beta$ is an equivalence of tensor categories, the
property of  having such a homomorphism $p$ is preserved; so (\ref{alg1}) follows.   Part
(\ref{alg2}) is similar, since a multiplicative identity corresponds to an $F$-vector
space homomorphism $i:F \to A$ satisfying the identity law. 

Part (\ref{alg3}) follows from the fact that up to isomorphism, $\beta$ has an inverse given by
intersection (i.e.\ fibre product or inverse limit); see Propositions~\ref{harb} and \ref{patchprob}.  
So a commutative
$F$-algebra induces commutative algebras on the patches and vice versa.   
Part (\ref{alg4}) holds because if $F'$ is a field extension of $F$, then a finite
$F$-algebra $E$ is separable if and only if the $F'$-algebra $E \otimes_F F'$ is separable.  

For part (\ref{alg5}), the first condition (on $G$-invariants) follows using that the inverse to $\beta$ is
given by intersection, together with the fact that the intersection of the rings
of $G$-invariants in fields $E_i:= E \otimes_F F_i$ is the ring of $G$-invariants in the
intersection of the $E_i$.  The second condition, on inertia groups, holds because the residue fields of $E$ are contained in those of $E_i$, with the $G$-actions on the latter being induced by those on the former.

For part (\ref{alg6}), we are reduced by (\ref{alg2}) to verifying that centrality and
simplicity are preserved.  If $E$ is the center of an $F$-algebra $A$, then $E
\otimes_F F'$ is the center of the $F'$-algebra $A' := A \otimes_F F'$.  So
centrality is preserved by $\beta$ and its inverse.  The same holds for simplicity (in the presence of centrality) by \cite{pierce}, Section~12.4, Lemma~b.
\end{proof}

On the other hand, Theorem~\ref{patchingalgebras} as phrased above does {\it not} apply to
(finite dimensional central) division algebras over $F$.  For example, in the
context of global patching in Section~\ref{globalsection}, let $T = k[[t]]$ where $\ch k \ne 2$; $\hat
X = \P^1_T$ (the projective $x$-line over $T$); $U_1 = \A^1_k = \P^1_k \smallsetminus
\{\infty\}$, $U_2 = \P^1_k \smallsetminus \{0\}$, and $U_0 = U_1 \cap U_2 = \P^1_k \smallsetminus
\{0,\infty\}$.  With notation as in Section~\ref{globalsection}, we consider the function field $F =
k((t))(x)$ of $\hat X$, along with the fraction fields $F_1$, $F_2$, $F_0$ of the
rings $k[x][[t]]$, $k[x^{-1}][[t]]$, $k[x,x^{-1}][[t]]$, respectively.  Let $D$
be the quaternion algebra over $F$ generated by elements $a,b$ satisfying
$a^2=b^2=1-xt$, $ab=-ba$.  Then $D \otimes_F F_1$ is split as an algebra over
$F_1$, i.e.\ is isomorphic to $\Mat_2(F_1)$ (and not to a division
algebra), because $F_1$ contains an element $f$ such that $f^2=1-xt$ (where $f$
is given by the binomial power series expansion in $t$ for $(1-xt)^{1/2}$).   

But the other direction of the above theorem does hold for division algebras:
viz.\ if $D_1,D_2,D_0$ are division algebras over $F_1,F_2,F_0$ in the context
of Theorem~\ref{globalpatching}, then the resulting finite dimensional central
simple $F$-algebra $D$ (given by part (\ref{alg6}) of the above theorem) is in fact a
division algebra.  This is because $D$ is contained in the division algebras
$D_i$, hence it has no zero-divisors, and so is a division algebra
(being finite dimensional over $F$). 

Despite the failure of the above result for division algebras, below we state a
patching result for Brauer groups.  For any field $F$, let $\Br(F)$ be
the set of isomorphism classes of (finite dimensional central) division algebras
over $F$.  The elements of $\Br(F)$ are in bijection with the set of {\bf Brauer equivalence classes} $[A]$ of (finite dimensional) central simple $F$-algebras $A$.
Namely, by Wedderburn's theorem, every central simple $F$-algebra $A$ is
isomorphic to a matrix ring $\Mat_n(D)$ for some unique positive integer
$n$ and some $F$-division algebra $D$ which is unique up to isomorphism; and two
central simple algebras are called {\bf Brauer equivalent} if the underlying
division algebras are isomorphic.  By identifying elements of $\Br(F)$ with
Brauer equivalence classes, $\Br(F)$ becomes an abelian group under the
multiplication law $[A][B]=[A\otimes_F B]$, called the {\bf Brauer group} of $F$. (See also Chapter~4 of \cite{herstein}.) 

If $F'$ is an extension of a field $F$ (not necessarily algebraic), and if $A$
is a central simple $F$-algebra, then $A \otimes_F F'$ is a central simple
$F'$-algebra (\cite{pierce}, 12.4, Proposition~b(ii)).  Moreover if $A,B$ are Brauer equivalent over $F$, then $A
\otimes_F F', B \otimes_F F'$ are Brauer equivalent over $F'$.  So there is an
induced homomorphism $\Br(F) \to \Br(F')$.  In terms of this homomorphism, we can
state the following patching result for Brauer groups, which says that giving a
division algebra over a function field $F$ is equivalent to giving compatible
division algebras on the patches:

\begin{thm}
Under the hypotheses of Theorem~\ref{globalfactgeneral}, let $U = U_1 \cup U_2$ and form the fibre product of groups $\Br(F_1) \times_{\Br(F_0)} \Br(F_2)$ with respect to the maps $\Br(F_i) \to \Br(F_0)$ induced by $F_i \hookrightarrow F_0$.  Then
the base change map 
$\beta:\Br(F_U) \to \Br(F_1) \times_{\Br(F_0)} \Br(F_2)$ is a group isomorphism.
\end{thm}

\begin{proof}
Base change defines a homomorphism $\beta$ as above, and we wish to show that it is an isomorphism.  

For surjectivity, consider an element in $\Br(F_1) \times_{\Br(F_0)} \Br(F_2)$,
represented by a triple $(D_1,D_2,D_0)$ of division algebras over $F_1,F_2,F_0$
such that the natural maps $\Br(F_i) \to \Br(F_0)$ take the class of $D_i$ to
that of $D_0$, for $i=1,2$.   
Since the dimension of a division algebra is a square, there are 
positive integers $n_0,n_1,n_2$ such that the three integers $n_i^2\,{\rm
  dim}_{F_i}D_i$ (for $i=0,1,2$) are equal.  Let $A_i = \Mat_{n_i}(D_i)$ for $i=0,1,2$.
Then $A_i$ is a central simple algebra in the class of $D_i$ for $i=0,1,2$; and
$A_i \otimes_{F_i} F_0$ is $F_0$-isomorphic to $A_0$ for $i=1,2$, compatibly
with the inclusions $F_i \hookrightarrow F_0$ (because they lie in the same
class and have the same dimension).  So by part~(\ref{alg6}) of Theorem~\ref{patchingalgebras}, there
is a (finite dimensional) central simple $F_U$-algebra $A$ that induces
$A_0,A_1,A_2$ compatibly with the above inclusions. The class of $A$ is then an
element of $\Br(F_U)$ that maps under $\beta$ to the given element of $\Br(F_1)
\times_{\Br(F_0)} \Br(F_2)$. 

To show injectivity, consider an element in the kernel, represented by an
$F_U$-division algebra $D$.  Then $A_i := D \otimes_F F_i$ is split for $i=0,1,2$;
i.e.\ for each $i$ there is an $F_i$-algebra isomorphism $\psi_i:{\rm
  Mat}_n(F_i) \to A_i$, where $n^2 = {\rm dim}_F\, D$.  For $i=1,2$ let
$\psi_{i,0}$ be the induced isomorphism  
$\Mat_n(F_0) \to A_0$ obtained by tensoring $\psi_i$ over $F_i$ with $F_0$
and identifying each $A_i \otimes_{F_i} F_0$ with $A_0$.  So $\psi_{2,0}^{-1} \circ
\psi_{1,0}$ is an $F_0$-algebra automorphism of $\Mat_n(F_0)$, and hence is
given by (right) conjugation by a matrix $C \in \GL_n(F_0)$ (by \cite{herstein}, Corollary to Theorem~4.3.1).  By Theorem~\ref{globalfactgeneral},
there are matrices $C_i \in
\GL_n(F_i)$ such that $C = C_1C_2$.  Let $\psi_1' = \psi_1\rho_{C_1^{-1}}: {\rm
  Mat}_n(F_1) \iso A_1$ and $\psi_2' = \psi_2\rho_{C_2}: \Mat_n(F_2) \iso
A_2$, where $\rho_B$ denotes right conjugation by a matrix $B$.  Also let
$\psi_{i,0}':\Mat_n(F_0) \iso A_0$ be the isomorphism induced from
$\psi_i'$ by base change to $F_0$.  Then  
$\psi_{2,0}'^{-1} \circ \psi_{1,0}' = \rho_{C_2^{-1}} \rho_C \rho_{C_1^{-1}}$ is
the identity on $\Mat_n(F_0)$; the common isomorphism
$\psi_{1,0}' = \psi_{2,0}'$ will be
denoted by $\psi_0'$.  Thus the three isomorphisms
$\psi_i': \Mat_n(F_i) \iso A_i$ (for $i=0,1,2$) are compatible with the
natural isomorphisms $\Mat_n(F_i) \otimes_{F_i} F_0 \iso \Mat_n(F_0)$
and $A_i  \otimes_{F_i} F_0 \iso  A_0$ for $i=1,2$.  Equivalently, letting ${\rm
  CSA}(K)$ denote the category of finite dimensional central simple $K$-algebras
for a field $K$, the triples $(A_1,A_2,A_0)$ and $(\Mat_n(F_1),{\rm
  Mat}_n(F_2),\Mat_n(F_0))$, along with the associated natural base change
isomorphisms as above, represent isomorphic objects in the category ${\rm
  CSA}(F_1) \times_{{\rm CSA}(F_0)} {\rm CSA}(F_2)$.  Using the equivalence of
categories in part (vi) of the above theorem, there is up to isomorphism a
unique central simple $F_U$-algebra inducing these objects.  But $D$ and ${\rm
  Mat}_n(F_U)$ are both such algebras. Hence they are isomorphic.  So $n=1$ and
$D=F_U$, as desired.   
\end{proof}

These ideas are pursued further in \cite{hhk}, in the context of studying Galois groups of maximal subfields of division algebras.

%-------------------------------------------------------------------------------
\subsection{Inverse Galois Theory} \label{finitegroups}
%-------------------------------------------------------------------------------
We can use our results on patching over fields to recover results in inverse
Galois  theory that were originally proven (by the first author and others)
using patching over rings.  The point is that if $F$ is the fraction field of a
ring $R$,  then Galois field extensions of $F$ are in bijection with irreducible
normal  Galois branched covers of $\Spec R$, by considering generic fibres and
normalizations.  So one can pass back and forth between the two situations.

In particular, we illustrate this by proving the result below, on realizing Galois groups over the function field of the line over a complete discrete valuation ring $T$.  This result was originally shown in \cite{GCAL} (Theorem~2.3 and Corollary~2.4) using formal patching, and afterwards reproven in \cite{liu} using rigid patching.   We first fix some notation and terminology.  

Let $G$ be a finite group, let $H$ be a subgroup of $G$, and let $E$ be an
$H$-Galois  $F$-algebra for some field $F$.  The {\bf induced} $G$-Galois
$F$-algebra  $\Ind_H^G E$ is defined as follows:  

Fix a set $C = \{c_1,\dots,c_m\}$ of left coset representatives of $H$ in
$G$, with the identity coset being represented by the identity element.   
Thus for every $g \in G$ and every $i\in\{1,\ldots,m\}$ there is a unique $j$ such that
$gc_j \in c_iH$.   Let $\sigma^{(g)} \in S_m$ be the associated permutation
given by  $\sigma^{(g)}_i=j$.  Thus for each $i$, the element $h_{i,g} :=
c_i^{-1}gc_{\sigma^{(g)}_i}$ lies in $H$.

As an $F$-algebra, let $\Ind_H^G E$ be the direct product of $m$ copies of $E$
indexed by $C$.  For $g \in G$ and $(e_1,\dots,e_m) \in \Ind_H^G E$, set $g \cdot
(e_1,\dots,e_m) \in \Ind_H^G E$ equal to the element whose $i$th entry is $h_{i,g}(e_{\sigma^{(g)}_i})$.  This defines a $G$-action on $\Ind_H^G E$, whose fixed ring is $F$ (embedded diagonally).
For all $i,j \in \{1,\dots,m\}$, the elements of $c_iHc_j^{-1}$ define isomorphisms $E_j \to
E_i$,  where $E_i$ denotes the $i$th factor of $\Ind_H^G E$.  In particular,
$c_iHc_i^{-1}$  is the stabilizer of $E_i$ for each $i$.  One checks that up to
isomorphism,  this construction does not depend on the choice of left coset representatives.  

Note that $\Ind_1^G F$ is just the direct product of copies of $F$ 
that are indexed by $G$ and are permuted according to the left regular representation; i.e.\
$g \cdot (e_1,\dots,e_n) = (e_1',\dots,e_n')$ is given by $e_i' = e_j$ where
$gc_j=c_i$.   (Here $n = |G|$.)  Also, $\Ind_G^G E = E$ if $E$ is a $G$-Galois
$F$-algebra.   If $H \leq J \leq G$ and $E$ is an $H$-Galois $F$-algebra,
we may identify $\Ind_J^G \,\Ind_H^J E$ with $\Ind_H^G E$ as $G$-Galois
$F$-algebras.   If $A$ is any $G$-Galois $F$-algebra, and $E$ is a
maximal  subfield of $A$ containing $F$, then $E$ is a Galois field extension of
$F$ whose Galois group $H := \Gal(E/F)$ is a subgroup of $G$, and $A$ is
isomorphic to  $\Ind_H^G E$ as a $G$-Galois $F$-algebra.

As in the proof in \cite{GCAL}
of the result below, we will patch together ``building blocks'' which are Galois and cyclic and which induce trivial extensions over the closed fibre $t=0$ (though here we will consider extensions of fields rather than rings).  For example, if $F$ contains a primitive $n$th root of unity, then an $n$-cyclic building block may be given by $y^n=f(f-t)^{n-1}$, for some $f$.  If there is no primitive $n$th root of unity in $F$ but $n$ is prime to the characteristic, then one can descend some $n$-cyclic extension of the above form from $F[\zeta_n]$ to $F$; while if $n$ is a power of the characteristic, building blocks can be constructed using Artin-Schreier-Witt extensions. See \cite{GCAL}, Lemma~2.1, for an explicit construction.

\begin{thm}
Let $K$ be the fraction field of a complete discrete valuation ring $T$ and let
$G$ be a finite group.  Then $G$ is the Galois group of a Galois field extension
$A$ of $K(x)$ such that $K$ is algebraically closed in $A$.
\end{thm}

\begin{proof}
Let $g_1,\dots,g_r$ be generators for $G$ that have prime power orders, and let $H_i \leq G_i$ be the cyclic subgroup generated by $g_i$.  Let $k$ be the residue field of $T$, and pick distinct
monic irreducible polynomials $f_1(x),\dots,f_n(x) \in k[x]$; these define distinct 
closed points $P_1,\dots,P_n$ of the projective $x$-line $\P^1_k$.  For each~$i$ let $\hat f_i(x) \in T[x]$ be
some (irreducible) monic polynomial lying over $f_i(x)$.  This defines
a lift $\hat P_i$ of $P_i$ to a reduced effective divisor on $\P^1_T$.  

According to \cite{GCAL}, Lemma~2.1, there is an irreducible $H_i$-Galois branched
cover  $Y_i \to \P^1_T$ whose special fibre is unramified away from $P_i$, and
such  that its fibre over the generic point~$\eta$ of the special fibre is trivial
(corresponding  to a mock cover, in the terminology there).  
That is, there is an isomorphism 
$\phi_i:\Spec\bigl(\Ind_1^{H_i} k(x)\bigr) \to Y_i \times_{\P^1_T} \eta $ as $H_i$-Galois covers of~$\eta$.
Replacing $Y_i$ by
its  normalization in its function field, we may assume that $Y_i$ is normal.
Necessarily, $Y_i \to \P^1_T$ is totally ramified over the closed point $P_i$.
Namely, if $I \leq H_i$ is the inertia group at $P_i$ then $Y_i/I \to \P^1_T$
is unramified and hence purely arithmetic (i.e.\ of the form $\P^1_S \to \P^1_T$
for some finite extension $S$ of $T$); but generic triviality on the special
fibre  then implies that $S=T$ and so $I=H_i$.  (The fact that it is totally
ramified  at $P_i$ can also be deduced from the explicit expressions in the
proof of \cite{GCAL}, Lemma~2.1.)  

Let $t$ be a uniformizer for $T$, and for $i=1,\dots,r$ let $\hat R_i$ be the
$t$-adic  completion of the local ring of $\P^1_T$ at $P_i$, with fraction field
$F_i$.  The pullback of $Y_i \to \P^1_T$ to $\Spec \hat R_i$ is finite and
totally ramified.  Hence it is irreducible, of the form $\Spec \hat S_i$ for
some finite extension~$\hat S_i$ of~$\hat R_i$ that is a domain.  Thus the
fraction field $E_i$ of $\hat S_i$ is an $H_i$-Galois field extension of~$F_i$.
Let $\hat R_0$ be the completion of the local ring of $\P^1_T$ at $\eta$, with fraction field $F_0$.  
Let~$R_{r+1}$ be the subring of $F:=K(x)$ consisting of the rational functions on $\P^1_T$ that are
regular on the special fibre $\P^1_k$ away from $P_1,\dots,P_r$.  Let $\hat
R_{r+1}$ be the $t$-adic completion of~$R_{r+1}$ and let~$F_{r+1}$ be the
fraction field of $\hat R_{r+1}$.  
Also let $H_0 = H_{r+1} = 1 \leq G$ and write $E_0=F_0$, $E_{r+1} = F_{r+1}$.  For $i=0,1,\dots,r+1$, we consider the
$G$-Galois $F_i$-algebra $A_i := \Ind_{H_i}^G E_i$.

We claim that there is 
an isomorphism $E_i \otimes_{F_i} F_0 \to \Ind_1^{H_i}\, F_0$ of $H_i$-Galois $F_0$-algebras along with compatible $F_i$-algebra inclusions 
$E_i \hookrightarrow E_0 = F_0$ and $A_i \hookrightarrow A_0$, for $i=1,\dots,r+1$.  In the case $i=r+1$ this is clear from the definitions of $H_{r+1}$ and $E_{r+1}$, via the inclusion $F_{r+1} \hookrightarrow F_0$.  For $1 \le i \le r$, the asserted isomorphisms are induced by $\phi_i$.  Namely, 
by Hensel's Lemma applied to $\hat R_0$, there is a unique isomorphism
$\hat\phi_i:\Spec\bigl(\Ind_1^{H_i}\, \hat R_0\bigr) \to 
Y_i \times_{\P^1_T} \Spec \hat R_0$ of $H_i$-Galois covers of $\Spec \hat R_0$ that lifts $\phi_i$.  Using the natural identifications
$Y_i \times_{\P^1_T} \Spec \hat R_0
= Y_i \times_{\P^1_T} \Spec \hat R_i \times_{\Spec \hat R_i} \Spec \hat R_0 = \Spec \hat S_i \times_{\Spec \hat R_i} \Spec \hat R_0 = \Spec(\hat S_i \otimes_{\hat R_i} \hat R_0)$,
the isomorphism $\hat\phi_i$ corresponds on the 
ring level to an isomorphism $\hat S_i \otimes_{\hat R_i} \hat R_0 \to \Ind_1^{H_i}\, \hat R_0$ of $H_i$-Galois $\hat R_0$-algebras.  Since $E_i$ is the fraction field of the finite $\hat R_i$-algebra $\hat S_i$, and since $F_i$ is the fraction field of $\hat R_i$, there is a natural identification of $E_i$ with $\hat S_i \otimes_{\hat R_i} F_i$.  So tensoring the above $\hat R_0$-algebra isomorphism with  $F_0$ yields an isomorphism $E_i \otimes_{F_i} F_0 = \hat S_i \otimes_{\hat R_i} F_i \otimes_{F_i} F_0 = \hat S_i \otimes_{\hat R_i} F_0  \to \Ind_1^{H_i}\, F_0$ of $H_i$-Galois $F_0$-algebras; and
hence also an $F_i$-algebra inclusion $E_i \hookrightarrow F_0$, using the projection onto the identity component.  
The functor $\Ind_{H_i}^G$ then induces an isomorphism $\Ind_{H_i}^G\,(E_i \otimes_{F_i} F_0) \to \Ind_{H_i}^G\,\Ind_1^{H_i}\, F_0 = \Ind_1^G F_0 = A_0$ of $G$-Galois $F_0$-algebras.  Tensoring the inclusion $F_i \hookrightarrow F_0$ with the 
$F_i$-algebra~$A_i$ yields an $F_i$-algebra inclusion
$A_i = \Ind_{H_i}^G\,E_i \hookrightarrow (\Ind_{H_i}^G\,E_i) \otimes_{F_i} F_0 = \Ind_{H_i}^G\,(E_i \otimes_{F_i} F_0) \to A_0$, concluding the verification of the claim.

Thus we may apply Theorem~\ref{patchingalgebras}(\ref{alg5}), in the case of Theorem~\ref{globalpatchingseveral}, 
to the fields~$F_i$ and the
$G$-Galois $F_i$-algebras $A_i$, for $i=0,1,\dots,r+1$.  We then obtain a $G$-Galois
$F$-algebra~$A$ that induces the~$A_i$'s compatibly.  Moreover, 
as observed after Theorem~\ref{globalpatchingseveral}, $A$ is
the intersection of the algebras $A_1,\dots,A_r,A_{r+1}$ inside $A_0$.  Note that $K$ is algebraically closed in $A$ because it is algebraically closed in $F_0$ and hence in $A_0$.

It remains to show that the $G$-Galois $F$-algebra $A$ is a field. 
For $i=0,1,\dots,r+1$ let $I_i \subset A_i$ be the kernel of the projection of $A_i = \Ind_{H_i}^G E_i$ onto the identity copy of $E_i$ (i.e.\ the copy of $E_i$ indexed by the identity element of $G$), and identify this identity copy with $A_i/I_i$.  
Then $I_i \subset A_i$ is the inverse image
of $I_0 \subset A_0$ under $A_i \hookrightarrow A_0$, since the inclusions $A_i \hookrightarrow A_0$ and $E_i \hookrightarrow E_0 = F_0$ are compatible with the projections $A_i \to E_i$ onto their identity components.
Let $I
\subseteq A$ be the inverse image of $I_0 \subset A_0$ under $A
\hookrightarrow A_0$, and let $E=A/I$.  Thus~$I$ is also the inverse image of $I_i \subset A_i$ under $A \hookrightarrow A_i$, since $A \hookrightarrow A_0$ factors
through~$A_i$ and $I_i \subset A_i$ is the inverse image
of $I_0 \subset A_0$ under $A_i \hookrightarrow A_0$.  Hence~$I$ is a prime ideal of~$A$, and~$E$ is an integral domain.  But~$E$ is finite over the field $F$, since the $G$-Galois $F$-algebra~$A$ is.  Thus~$E$ is a field.  Now the above inclusions are compatible with the $G$-Galois actions. So using the identification $E_i = A_i/I_i$ and the fact that $H_i = \Gal(E_i/F_i) \subseteq G$ is the stabilizer of $I_i$ in~$G$, we have that every element of~$H_i$ restricts to an element of $H := \Gal(E/F) \subseteq G$, the stabilizer of~$I$ in~$G$.  That is, $H$ contains~$H_i$ for all~$i$.  But $H_1,\dots,H_r$ generate~$G$.  So $H=G$.  Thus~$I$ is stabilized by all of~$G$; and since the identity component of each element of~$I$ is zero (regarding $I \subseteq I_0 \subset A_0 = \Ind_1^G F_0$), it follows that $I=(0)$.  Hence $E=A$ and~$A$ is a field. 
\end{proof}

\begin{remark}
\renewcommand{\theenumi}{\alph{enumi}}
\begin{enumerate}

\item The above proof can be extended to more general smooth curves $\hat X$ over a complete discrete valuation ring $T$.  Namely, Theorem~\ref{globalpatchingseveral} permits patching on such curves; and the same expressions used for building blocks in the case of the line can be used for other curves, since they remain $n$-cyclic and totally ramified.  This latter fact can be seen directly from the construction in \cite{GCAL}.  It can also be seen by choosing a parameter $x$ for a point $P$ on the closed fibre $X$ of $\hat X$; constructing the building blocks for the $x$-line over $T$; and then taking a base change to the local ring at $P$ (which, being \'etale, preserves total ramification).  This contrasts with the strategy in \cite{haranjarden2}, Proposition~1.4, which is to map a curve to the line; perform a patching construction there; and then deduce a result about the curve.\label{rem1}

\item Alternatively, the above proof can be extended to more general smooth curves over $T$ by using Theorem~\ref{localpatchingseveral} instead of Theorem~\ref{globalpatchingseveral} (where the complete local ring is independent of which smooth curve is taken).  It can also be extended to the case of a singular normal $T$-curve whose closed fibre is generically smooth, by instead using Theorem~\ref{singularpatching}.\label{rem2}

\item In \cite{GCAL}, Section 2, more was shown: that the theorem remains true if we replace
$T$ by any complete local domain that is not a field.  But in fact this more
general assertion follows from the above theorem because every such domain
contains a complete discrete valuation ring; see \cite{jarden}, Lemma~1.5 and Corollary~1.6. \label{rem3}
  
\item One can similarly recover other results in inverse Galois theory within our framework of patching over fields; e.g., the freeness of the absolute Galois group of $k(x)$, for $k$ algebraically closed (the ``Geometric Shafarevich Conjecture'' \cite{harep}, \cite{popep}).  But the above result is merely intended to be illustrative, to show how patching over fields can be used in geometric Galois theory.\label{rem4}
 
\end{enumerate}
\end{remark}

%-------------------------------------------------------------------------------
\subsection{Differential Modules} \label{differential}
%-------------------------------------------------------------------------------
The main interest in patching vector spaces is of course that we can also patch vector spaces with additional structure. 
This was done for various types of algebras in Section~7.1 above.
The following application is another example of this sort. 

Suppose that $F$ is a field of characteristic zero 
equipped with a derivation $\partial_F$. A {\bf differential module}
over $F$ is a finite dimensional $F$-vector space $M$ together with an additive map
$\partial_M:M\rightarrow M$ such that $\partial_M(f\cdot m)=\partial_F(f)\cdot m + f\cdot \partial_M(m)$ 
for all $f \in F, m \in M$ (Leibniz rule). A {\bf homomorphism of differential modules} is a homomorphism 
of the underlying vector spaces that respects the differential structures. 
It is well known that differential modules over a differential field $F$ form a tensor category 
$\DMod (F)$ (in fact a Tannakian category over $F$; e.g. see \cite{matzatvdp}, \S1.4). 

We will state only the simplest version of patching differential modules, a
consequence of Theorem~\ref{globalpatching}. There are respective versions of
Theorem~\ref{globalpatchingseveral}, and of the patching results in
Section~\ref{localsection}. 

\begin{thm}\label{diffmodpatching}
Let $T$ be a complete discrete valuation ring with fraction field $K$ of characteristic zero and residue 
field $k$, and let $\hat X$ be a smooth connected projective $T$-curve with
closed fibre $X$ and function field $F$.
Let $U_1, U_2\subseteq X$, and let $U:=U_1 \cup U_2$, $U_0:=U_1 \cap U_2$. 
Equip $F_U,F_{U_i}$ with the derivation $\frac{d}{dx}$ for some rational function $x$ on $\hat X$ that is not contained in $K$.

Then the base change functor
$$\DMod (F_U)\rightarrow \DMod
(F_{U_1})\times_{\DMod (F_{U_0})} 
\DMod (F_{U_2})$$
is an equivalence of categories, with inverse given by intersection.
\end{thm}

\begin{proof} 
Recall that $F_U=F_{U_1}\cap F_{U_2}$ (Theorem~\ref{globalintersection}). 
By Theorem~\ref{globalpatching}, base change is an equivalence of categories on
the level of vector spaces; so for every object $(M_1,M_2;\phi)$ in $\DMod
(F_{U_1})\times_{\DMod (F_{U_0})} \DMod
(F_{U_2})$, there is an $F_U$-vector space $M$ that induces $(M_1,M_2;\phi)$
as an object in $\Vect(F_{U_1})\times_{\Vect(F_{U_0})} \Vect(F_{U_2})$.  Moreover, as noted after that result, $M$ is given by $M_1 \cap M_2$. Consequently, the derivations on $M_1$ and $M_2$ restrict compatibly to $M$; i.e., $M$ is a differential module under that common restricted derivation.  By Corollary~\ref{dimcriterion}, $\dim_{F_U} M = \dim_{F_{U_i}} M_i$ for $i=1,2$; in particular, $M$ contains a basis of $M_i$ as a vector space over $F_{U_i}$ ($i=1,2$). But the derivation on each $M_i$ is already determined when given on such a basis (by the Leibniz rule). Thus $M$ induces the $M_i$'s as differential modules, compatibly with $\phi$.

So the base change functor gives a bijection on isomorphism classes.  Similarly, morphisms between corresponding objects in the two categories are in bijection on the level of vector spaces, and hence also on the level of differential modules (using that the derivations are related by taking
base change and restriction).  Thus the functor is an equivalence of categories. \end{proof}

After choosing a basis of each $M_i$ in the above proof, one can also explicitly define the derivation on $M$ using the matrix representations of the derivations and a factorization of the matrix defining $\phi$ given by Theorem~\ref{globalfactgeneral}.

\begin{remark}
As noted in the proof of Theorem~\ref{patchingalgebras}, the equivalence of the
categories of vector spaces is in fact an equivalence of tensor categories; the
same remains true for differential modules.
\end{remark}

There is a Galois theory for differential modules that mimics the usual Galois theory of finite field extensions.  A natural question to ask is whether one can
control the differential Galois group of a differential module obtained by patching. This question
(along with its implications for the inverse problem in differential Galois theory) is
the subject of \cite{harhar2} (see also \cite{owrep2}), which provides applications of the above theorem.

%%%%%%%%%%%%%%%%%%%%%%%%%%%%%%%%%%%%%%%%%%%%%%%%%%%%%%%%%%%%%%%%%%%%%%%%%%%%%%%%

\medskip

\noindent Author information:

\medskip

\noindent David Harbater: Department of Mathematics, University of Pennsylvania, Philadelphia, PA 19104-6395, USA; email: {\tt harbater@math.upenn.edu}

\medskip

\noindent Julia Hartmann: IWR, University of Heidelberg, Im Neuenheimer Feld 368, 69120 Heidelberg, Germany; email:  {\tt 
Julia.Hartmann@iwr.uni-heidelberg.de}

\end{document}